\documentclass[12pt,a4paper, draft]{article}
\usepackage[dvipdfm, margin=2cm]{geometry}
\usepackage[tbtags]{amsmath}%
\usepackage{amsthm}
\usepackage{amssymb}%
\usepackage{amsfonts}
\usepackage{mathrsfs}
\usepackage{color}
\allowdisplaybreaks[4]
\numberwithin{equation}{section}
\newtheorem{theorem}{Theorem}[section]
\newtheorem{proposition}[theorem]{Proposition}

\newtheorem{lemma}[theorem]{Lemma}
\theoremstyle{definition}
\newtheorem{remark}[theorem]{Remark}

\newtheorem{definition}[theorem]{Definition}

\DeclareMathOperator{\ad}{ad}

\DeclareMathOperator{\End}{End}

\DeclareMathOperator{\Rad}{Rad}
\DeclareMathOperator{\Res}{Res}

\DeclareMathOperator{\wt}{wt}

\newcommand{\BF}{\mathbb{F}}
\newcommand{\BN}{\mathbb{N}}
\newcommand{\BZ}{\mathbb{Z}}
\newcommand{\BQ}{\mathbb{Q}}
\newcommand{\BC}{\mathbb{C}}

\newcommand{\spanf}{\mathrm{span}_\BF}

\newcommand{\fg}{\mathfrak{g}}
\newcommand{\fh}{\mathfrak{h}}
\newcommand{\D}{\mathcal{D}}

\newcommand{\B}{\mathcal{B}}

\newcommand{\bk}{\mathbf{k}}

\newcommand{\1}{\mathbf{1}}

\newcommand{\?}{\hspace*{1.346em}}
\makeatletter
\newcommand{\leqnomode}{\tagsleft@true}
\newcommand{\reqnomode}{\tagsleft@false}
\makeatother
\makeatletter
\newtoks\@enLab  
\def\@enQmark{?}
\def\@enLabel#1#2{%
	\edef\@enThe{\noexpand#1{\@enumctr}}%
	\@enLab\expandafter{\the\@enLab\csname the\@enumctr\endcsname}%
	\@enloop}
\def\@enSpace{\afterassignment\@enSp@ce\let\@tempa= }
\def\@enSp@ce{\@enLab\expandafter{\the\@enLab\space}\@enloop}
\def\@enGroup#1{\@enLab\expandafter{\the\@enLab{#1}}\@enloop}
\def\@enOther#1{\@enLab\expandafter{\the\@enLab#1}\@enloop}
\def\@enloop{\futurelet\@entemp\@enloop@}
\def\@enloop@{%
	\ifx A\@entemp         \def\@tempa{\@enLabel\Alph  }\else
	\ifx a\@entemp         \def\@tempa{\@enLabel\alph  }\else
	\ifx i\@entemp         \def\@tempa{\@enLabel\roman }\else
	\ifx I\@entemp         \def\@tempa{\@enLabel\Roman }\else
	\ifx 1\@entemp         \def\@tempa{\@enLabel\arabic}\else
	\ifx \@sptoken\@entemp \let\@tempa\@enSpace         \else
	\ifx \bgroup\@entemp   \let\@tempa\@enGroup         \else
	\ifx \@enum@\@entemp   \let\@tempa\@gobble          \else
	\let\@tempa\@enOther
	\fi\fi\fi\fi\fi\fi\fi\fi
	\@tempa}
\newlength{\@sep} \newlength{\@@sep}
\providecommand{\sfbc}{\rmfamily\upshape}
\providecommand{\sfn}{\rmfamily\upshape}
\def\@enfont{\ifnum \@enumdepth >1\let\@nxt\sfn \else\let\@nxt\sfbc \fi\@nxt}
\def\enumerate{%
	\ifnum \@enumdepth >3 \@toodeep\else
	\advance\@enumdepth \@ne
	\edef\@enumctr{enum\romannumeral\the\@enumdepth}\fi
	\@ifnextchar[{\@@enum@}{\@enum@}}
\def\@@enum@[#1]{%
	\@enLab{}\let\@enThe\@enQmark
	\@enloop#1\@enum@
	\ifx\@enThe\@enQmark\@warning{The counter will not be printed.%
		^^J\space\@spaces\@spaces\@spaces The label is: \the\@enLab}\fi
	\expandafter\edef\csname label\@enumctr\endcsname{\the\@enLab}%
	\expandafter\let\csname the\@enumctr\endcsname\@enThe
	\csname c@\@enumctr\endcsname7
	\expandafter\settowidth
	\csname leftmargin\romannumeral\@enumdepth\endcsname
	{\the\@enLab\hskip\labelsep}%
	\@enum@}
\def\@enum@{\list{{\@enfont\csname label\@enumctr\endcsname}}%
	{\usecounter{\@enumctr}\def\makelabel##1{\hss\llap{##1}}%
		\ifnum \@enumdepth>1\setlength{\topsep}{\@@sep}\else
		\setlength{\topsep}{\@sep}\fi
		\ifnum \@enumdepth>1\setlength{\itemsep}{0pt plus1pt minus1pt}%
		\else \setlength{\itemsep}{\@@sep}\fi
		\setlength{\parsep}{0pt plus1pt minus1pt}%
		\setlength{\parskip}{0pt plus1pt minus1pt}
}}

\def\endenumerate{\par\ifnum \@enumdepth >1\addvspace{\@@sep}\else
	\addvspace{\@sep}\fi \endlist}
\makeatother                                                            %

\begin{document}
\title{Lattice vertex algebras of type ADE over fields of prime characteristic and their representations}
\author{ Qiang Mu\thanks{Supported by the Heilongjiang Provincial NSF (grant PL2025A004).}\quad and \quad Hongju Zhao}
\date{School of Mathematical Sciences, Harbin Normal University,\\
Harbin 150025, China}
\maketitle

\begin{abstract}
We study lattice vertex algebras of type ADE over an algebraically closed 
field $\BF$ of prime characteristic $p>2$ and their representations.
Let $L$ be a root lattice of type ADE, and $G_{L}$ the Gram matrix of $L$.
When $\det G_{L}\not \equiv 0\pmod{p}$, we establish an isomorphism between 
the lattice vertex algebra $V_{L,\BF}$ and the level-one simple affine vertex 
algebra of the same type.
Via this isomorphism, we classify the irreducible $\BN$-graded modules of 
$V_{L,\BF}$ viewed as an $\BN$-graded vertex algebra.
We also consider the case where $L$ is of type $A_n$ with $\det G_L=n+1\equiv 0\pmod{p}$.
We show that $V_{L,\BF}$ is not simple and determine the simple $\BN$-graded 
quotient of $V_{L,\BF}$.
Furthermore, when $n+1=ap$ for some $a\in \BZ_+$ with $\gcd(a,p)=1$, 
we give the classification of the irreducible $\BN$-graded modules for 
the simple quotient of $V_{L,\BF}$.
\end{abstract}

\section{Introduction}

Vertex (operator) algebras associated with even lattices are among the 
most important objects in vertex algebra theory and provide many 
fundamental examples (see, for example,  \cite{D,DG,DLM,FLM,LW}).
For instance, the well-known Moonshine vertex operator algebra $V^{\natural}$ 
is constructed from the vertex operator algebra associated to the Leech 
lattice (see \cite{FLM}).

Let $L$ be a positive definite even lattice with a basis $\{\gamma^1,\ldots,\gamma^d\}$ 
and a bilinear form $\langle\cdot,\cdot\rangle$.
Over a field of characteristic zero, 
one can construct the lattice vertex operator algebra $V_L$
associated with $L$,
called the lattice vertex operator algebra.
However, this approach does not directly extend to fields of prime characteristic.
The integral form provides a common tool for passing from characteristic 
zero to prime characteristic.
For example, in \cite{DG}, Dong and Griess introduced an integral form of 
the lattice vertex operator algebra $V_{L}$, denoted by $V_{L,\BZ}$.
Subsequently, Mu studied lattice vertex algebras over fields $\BF$ of 
prime characteristic $p$ obtained from $V_{L,\BZ}$, 
namely $V_{L,\BF}:=V_{L,\BZ}\otimes_{\BZ}\BF$, 
and proved that if $\det G_L\not\equiv 0 \pmod{p}$, then $V_{L,\BF}$ is 
a simple vertex operator algebra (see \cite{Mu}). 
Here, $G_{L}=(\langle\gamma^i,\gamma^j\rangle)$ is the Gram matrix of $L$.

Note that the lattice vertex algebra of type ADE is isomorphic to the 
level-one simple affine vertex algebra of the same type over a field 
of characteristic zero (see, for example, \cite{Kac2}).
This isomorphism allows one to investigate the representation theory 
of lattice vertex algebras via that of affine vertex algebras.
Motivated by this, we study lattice vertex algebras of type ADE over 
fields of prime characteristic and their representations, 
building on the results of \cite{DG} and \cite{Mu}.

Let $\BF$ be an algebraically closed field of prime characteristic $p>2$.
Let $\fg$ be a Lie algebra of type ADE over $\BF$, and let $L$ be the 
root lattice of $\fg$.
We first consider the case where $\fg$ is simple; that is, $p\nmid n+1$ 
if $\fg$ is of type $A_n$, and $p>3$ if $\fg$ is of type $E_6$ (see \cite[\S 0.13]{H1995}).
Under these assumptions, $\det G_{L}\not\equiv 0\pmod{p}$.
Let $L_{\hat{\fg}}(1,0)$ be the corresponding level-one simple 
affine vertex algebra over $\BF$.  
We prove that
\begin{equation}\label{eq:iso-1}
V_{L,\BF}\cong L_{\hat{\fg}}(1,0)
\end{equation}
as vertex algebras (see Theorem \ref{th:iso-1}).
Using the isomorphism \eqref{eq:iso-1}, we then classify the irreducible 
$\BN$-graded modules of $V_{L,\BF}$ viewed as an $\BN$-graded vertex algebra.
Specifically, recall from \cite{JLM} that for $\ell\in \{0,1,\ldots,p-1\}$, 
$L_{\widehat{\mathfrak{sl}}(2,\BF)}(\ell,0)$ has at most $\ell +1$ 
irreducible $\BN$-graded modules.
We show that the simple vertex algebra $L_{\widehat{\fg}}(1,0)$ has 
at most $\det G_L$ irreducible $\BN$-graded modules, generalizing the 
result in \cite{JLM}.
On the other hand, since $\det G_{L}\not\equiv 0\pmod{p}$ and 
$|L^*/L|=\det G_{L}$, by \cite[Theorem 3.8]{ZM}, 
there exist $\det G_{L}$ inequivalent irreducible $\BN$-graded modules for $V_{L,\BF}$.
Hence, via the isomorphism \eqref{eq:iso-1}, 
these $\det G_{L}$ irreducible $\BN$-graded modules exhaust all 
the irreducible $\BN$-graded modules of $V_{L,\BF}$ (see Theorem \ref{th:clas-type-ADE}).

We now consider the case where $L$ is a root lattice of type $A_n$ 
with $\det G_{L}=n+1\equiv 0\pmod{p}$.
We begin by examining the simplicity of $V_{L,\BF}$.
Recall that in characteristic zero, there is a unique nondegenerate 
symmetric invariant bilinear form $(\cdot,\cdot)_{V_{L}}$ on the 
vertex operator algebra $V_{L}$ such that 
$(\1,\1)_{V_{L}}=1$ (see \cite{L}).
From this, a symmetric bilinear form $(\cdot,\cdot)_{V_{L,\BF}}$ on 
$V_{L,\BF}$ is defined accordingly (see \cite{Mu}).
Set $\Rad(\cdot,\cdot)_{V_{L,\BF}}:=\{u\in V_{L,\BF}\mid (u,v)_{V_{L,\BF}}=0
\text{ for all }v\in V_{L,\BF}\}$.
We prove that $\Rad(\cdot,\cdot)_{V_{L,\BF}}$ is a nontrivial ideal of 
$V_{L,\BF}$ (see Lemma \ref{lem:Rad}).
Using an argument analogous to that in \cite{Mu} (see also \cite[Proposition 3.6]{ZM}), 
we deduce that $V_{L,\BF}/\Rad(\cdot,\cdot)_{V_{L,\BF}}$ is a simple ($\BN$-graded) 
vertex algebra (see Lemma \ref{lem:simp-quot-latt}).
Moreover, we determine the maximal $\BN$-graded ideal $\Rad(\cdot,\cdot)_{V_{L,\BF}}$.
Concretely, since $p\mid n+1$, we write $n+1=ap^r$ for some $a,r\in \BZ_+$ with $\gcd(a,p)=1$.
Let $\Pi:=\{\alpha_1,\ldots,\alpha_n\}$ be the simple roots of $\mathfrak{sl}(n+1,\BF)$, 
which form a basis of $L$.
Set $\eta:=\alpha_1+2\alpha_2+\cdots+n\alpha_n$, and
let 
\begin{equation*}
	I_{L}=\spanf\{s_{\alpha^1,m_1}\ldots s_{\alpha^k,m_k}s_{\eta,m}\iota(e_{\alpha})\mid 
\alpha\in L,\alpha^i\in \Pi, m_i\in \BZ_+, k\in \BN, m\in \BZ_+\text{ with }p^r\nmid m\},
\end{equation*}
where  $s_{\gamma,m}$ is the coefficient of $x^{m}$ in $E^{-}(-\gamma,x)$ for 
$\gamma\in L$ (see \eqref{eq:def-E+E-}).
Then we show that $I_{L}\subset \Rad(\cdot,\cdot)_{V_{L,\BF}}$ (see Lemma \ref{lem:4.4}),
and we further prove that $\Rad(\cdot,\cdot)_{V_{L,\BF}}=I_{L}$ (see Theorem \ref{th:I-L}).

We next study the irreducible modules for the simple quotient 
$V_{L,\BF}/\Rad(\cdot,\cdot)_{V_{L,\BF}}$.
Recall from \cite{SF} that when $p\mid n+1$, the Lie algebra 
$\mathfrak{sl}(n+1,\BF)$ has a one-dimensional ideal $\BF I_{n+1}$, 
and $\mathfrak{psl}(n+1,\BF):=\mathfrak{sl}(n+1,\BF)/\BF I_{n+1}$ is simple, 
where $I_{n+1}$ is the identity matrix of size $n+1$.
We prove that
\begin{equation}\label{eq:iso-2}
V_{L,\BF}/\Rad(\cdot,\cdot)_{V_{L,\BF}}\cong L_{\widehat{\mathfrak{psl}}(n+1,\BF)}(1,0)
\end{equation}
as vertex algebras (see Theorem \ref{th:iso-2}).
Now assume that $r=1$, i.e., $n+1=ap$ for some $a\in \BZ_+$ with $\gcd(a,p)=1$.
 Using the isomorphism \eqref{eq:iso-2}, 
 we show that $V_{L,\BF}/\Rad(\cdot,\cdot)_{V_{L,\BF}}$ has exactly $a$
 inequivalent irreducible $\BN$-graded modules, namely, 
 the modules listed in \eqref{eq:irr-quo-mod} (see Theorem \ref{th:clas-type-A-quo-kp}).

The paper is organized as follows. 
In Section \ref{sec:2}, we recall the construction of the lattice vertex 
algebra $V_{L,\BF}$ and its modules, and provide some relevant results.
Section \ref{sec:3} is devoted to the case where $L$ is of type ADE with 
$\det G_{L}\not\equiv 0\pmod{p}$.
We classify the irreducible $\BN$-graded modules for $V_{L,\BF}$ in this case.
In Section \ref{sec:4}, we treat the case where $L$ is a root lattice of type 
$A_n$ with $n+1\equiv 0\pmod{p}$.
We show that $V_{L,\BF}$ is not simple and determine the maximal $\BN$-graded 
ideal of $V_{L,\BF}$.
Furthermore, when $n+1=ap$ for some $a\in \BZ_{+}$ with $\gcd(a,p)=1$, 
the classification of the irreducible $\BN$-graded modules for the simple 
quotient of $V_{L,\BF}$ is given.

In this paper, let $\BC$, $\BQ$, $\BZ$, $\BZ_+$, and $\BN$ denote the sets of 
complex numbers, rational numbers, integers, positive integers, 
and nonnegative integers, respectively.
Furthermore, let $\BF$ be an algebraically closed field of prime characteristic $p>2$.

\section{Lattice vertex algebras and their modules}\label{sec:2}

In this section, we review the construction of the vertex algebra $V_{L,\BF}$ 
over $\BF$ associated to a positive definite even lattice $L$ and its 
modules (see \cite{DG, Mu}). 
We also give some relevant results.

We will follow the conventions of \cite{FLM} and \cite{LL} for the lattice 
vertex operator algebra $V_L$ over $\BC$, the Heisenberg vertex operator 
subalgebra $M(1)$ of $V_{L}$, and $V_L$-modules.
Let $L$ be a positive definite even lattice with a basis 
$\{\gamma^1,\ldots,\gamma^d\}$ and a bilinear form $\langle\cdot,\cdot \rangle$.
Then $(V_L,Y,\mathbf{1},\omega)$ is a simple vertex operator algebra over 
$\BC$ of central charge $c=\operatorname{rank}L$. 
Here, $\mathbf{1}=1\otimes\iota(e_{0})$ is the vacuum vector.
The vertex operator is given by
\begin{align}
Y(h,x)&=h(x)=\sum_{m\in \BZ}h(m)x^{-m-1}\;\text{ for }h\in \fh,\\
Y(\iota(e_{\alpha}),x)&=E^{-}(-\alpha,x)E^{+}(-\alpha,x)e_{\alpha}x^{\alpha} \;
    \text{ for }\alpha\in L,\label{eq:ver-ope}
\end{align}
where $\fh=L\otimes_{\BZ}\BC$, and
\begin{equation}\label{eq:exp-E-E+}
E^{-}(-\alpha,x)=\exp\bigg(\sum_{m\in\BZ_+}\frac{\alpha(-m)}{m}x^{m}\bigg),
\quad E^{+}(-\alpha,x)=\exp\bigg(\sum_{m\in\BZ_+}\frac{-\alpha(m)}{m}x^{-m}\bigg).
\end{equation}
The conformal vector is given by
\begin{equation}\label{eq:conf-ve}
\omega=\frac{1}{2}\sum_{i=1}^{d}\gamma^{i}(-1)\beta^{i}(-1)\mathbf{1},
\end{equation}
where $\{\beta^{1},\ldots,\beta^d\}$ is a basis of $\fh$ 
dual to the basis $\{\gamma^1,\ldots, \gamma^d\}$ with respect to $\langle\cdot,\cdot\rangle$.

Let $L^*$ be the dual lattice of $L$, i.e., 
$L^*=\{\beta\in L\otimes_{\BZ} \BQ\mid \langle \beta,\alpha\rangle \in \BZ\text{ for all }\alpha\in L\}$.
For $\beta\in L^*$, 
one can similarly construct a $V_{L}$-module
\begin{equation}
V_{\beta+L}=S(\hat{\fh}_{-})\otimes \BC\{\beta+L\},
\end{equation}
where $\hat{\fh}_{-}=t^{-1}\fh[t^{-1}]$, 
and $\BC\{\beta+L\}=\operatorname{span}_{\BC}\{\iota(e_{\beta+\alpha})\mid \alpha\in L\}$.
Moreover, each irreducible $V_{L}$-module is isomorphic to 
$V_{\beta+L}$ for some $\beta+L\in L^*/L$ (see \cite[Theorem 3.1]{D}).

Following \cite{DG} and \cite{Mu}, for $\alpha\in L^*$, 
we write
\begin{equation}\label{eq:def-E+E-}
E^{-}(-\alpha,x)=\sum_{m\in \BN}s_{\alpha,m}x^{m},\quad   
E^{+}(-\alpha,x)=\sum_{m\in\BN}r_{\alpha,m}x^{-m}.
\end{equation}
In particular, $s_{\alpha,0}=r_{\alpha,0}=1$, 
$s_{\alpha,1}= \alpha(-1)$ and $r_{\alpha,1}=-\alpha(1)$.
We adopt the convention that $s_{\alpha,m}=r_{\alpha,m}=0$ for $m<0$.
For $\alpha,\beta\in L^*$,
we have (see, for example, \cite[(6.3.31)]{LL})
\begin{equation}\label{eq:E-ab=E-aE-b}
	E^{-}(-\alpha-\beta,x)=E^{-}(-\alpha,x)E^{-}(-\beta,x).
\end{equation}
For $\alpha,\beta\in L^*$, \cite[Proposition 6.3.14]{LL} gives
\begin{equation}\label{eq:E+E-E-E+}
E^{+}(-\alpha,x_1)E^{-}(-\beta,x_2)=\biggl(1-\frac{x_2}{x_1}\biggr)^{\langle \alpha,\beta\rangle}E^{-}(-\beta,x_2)E^{+}(-\alpha,x_1).
\end{equation}
Consequently, we obtain
\begin{equation}\label{eq:rs-sr}
r_{\alpha,n}s_{\beta,m}=\sum_{i\in \BN}(-1)^{i}\binom{\langle \alpha,\beta\rangle}{i}s_{\beta,m-i}r_{\alpha,n-i}
\end{equation}
for $m,n\in\BN$, as operators on $V_{L}$.

We recall the following result.

\begin{lemma}[{\cite[Theorem 3.3]{DG}}]\label{DG3.3}
Let
\begin{equation}\label{eq:M}
\mathcal{S}:=\big\{s_{\alpha^1,m_1}\cdots s_{\alpha^k,m_k}\bigm|
k\in\BN, \alpha^i\in\{\gamma^1,\ldots,\gamma^d\},
m_i\in\BZ_+, 1\le i\le k\big\}.
\end{equation}
When $k=0$, the product in \eqref{eq:M} is understood to be the
vacuum vector $\mathbf{1}$.
\begin{enumerate}[(i)]
\item The $\BZ$-span
\begin{equation*}
M(1)_{\mathbb Z}:=\operatorname{span}_{\BZ}\mathcal{S}
\end{equation*}
is an integral form of the vertex algebra $M(1)$.

\item Let $\mathcal{S}_L:=\{h\iota(e_\alpha)\mid h\in\mathcal{S}, \alpha\in L\}$.
Then the $\BZ$-span
\begin{equation*}
V_{L,\BZ}:=\operatorname{span}_{\BZ}\mathcal{S}_L
\end{equation*}
is an integral form of the vertex algebra $V_L$.
\end{enumerate}
\end{lemma}

It follows from Lemma~\ref{DG3.3} that
\begin{equation*}
V_{L,\BF}:=V_{L,\BZ}\otimes_{\BZ}\BF
\end{equation*}
is a vertex algebra over $\BF$, and that
\begin{equation*}
M(1)_{\BF}:=M(1)_{\BZ}\otimes_{\BZ}\BF
\end{equation*}
is a vertex subalgebra of $V_{L,\BF}$. Moreover, $\mathcal{S}$ and $\mathcal{S}_L$
form $\BF$-bases of $M(1)_{\BF}$ and $V_{L,\BF}$, respectively.
Furthermore, for $\beta\in L^*$, let $V_{\beta+L,\BZ}$
be the $\BZ$-span of $\{h\iota(e_{\beta+\alpha})\mid h\in\mathcal{S}, \alpha\in L\}$.
It follows from \cite[Theorem 4.3]{McRae} that $V_{\beta+L,\BZ}$ is 
an integral form of the $V_{L}$-module $V_{\beta+L}$ relative to $V_{L,\BZ}$,
and $V_{\beta+L,\BF}:=V_{\beta+L,\BZ}\otimes_{\BZ}\BF$ is a $V_{L,\BF}$-module.

The $\BN$-grading of $V_{L,\BZ}$ naturally induces a corresponding grading on $V_{L,\BF}$, i.e.,
\begin{equation}\label{eq:N-grad}
	V_{L,\BF}=\bigoplus_{m\in \BN}(V_{L,\BF})_{(m)},
\end{equation}
where $(V_{L,\BF})_{(m)}$ is spanned by elements 
$s_{\alpha^1,m_1} \cdots s_{\alpha^k,m_k} \iota(e_{\alpha})$
with $k \in \BN$, $m_i \in \BZ_+$, $\alpha^i \in\{\gamma^1,\ldots,\gamma^d\}$, $\alpha \in L$,
and $m_1 + \cdots + m_k + \frac{\langle \alpha, \alpha \rangle}{2} = m$.
One easily checks that  $V_{L,\BF}$ equipped with the $\BN$-grading 
\eqref{eq:N-grad} is an $\BN$-graded vertex algebra.
On the other hand, $V_{L,\BF}$ is $L$-graded, which is induced from 
the $L$-grading on $V_{L,\BZ}$ (see \cite[Theorem 4.3]{McRae}).
Thus, $V_{L,\BF}$ carries an $L\times\BN$-graded vertex algebra structure.

The following lemma was obtained in \cite{DG}. 

\begin{lemma}[{\cite[Lemma 3.2]{DG}}]\label{lem:DG-s--alpha}
For $\alpha\in L$ and $n\in \BN$, the element $s_{-\alpha,n}$ is a 
$\BZ$-linear combination of products of the elements $s_{\alpha,m}$
with $m \in \BN$ and $m \le n$.
\end{lemma}

\begin{theorem}[{\cite[Theorem 1]{Mu}}]\label{th:Mu-th-1}
$V_{L,\BF}$ as a vertex algebra is generated by $\{\iota(e_{\pm \gamma^i})\mid 1\le i\le d\}$.
\end{theorem}

For $\alpha\in L$, recall from \cite[(5)]{Mu} that
\begin{align*}
\sum_{n\in \BN}Y(s_{\alpha,n},x)x_0^n=E^{-}(-\alpha,x+x_0)E^{-}(\alpha,x)E^{+}(\alpha,x)E^{+}(-\alpha,x+x_0)\bigg(\frac{x+x_0}{x}\bigg)^{\alpha}
\end{align*}
on $V_{L,\BF}$.
For $n\in\BZ_+$, it follows that 
\begin{align}\label{eq:Y-s-alpha-n}
Y(s_{\alpha,n},x)&=\sum_{j=0}^{n}\sum_{i=0}^{n-j}\binom{\alpha }{j}x^{-j}\bigg(\sum_{k\in\BN }\binom{k}{n-j-i}s_{\alpha,k}x^{k-n+j+i}\bigg)E^{-}(\alpha,x)E^{+}(\alpha,x)\nonumber\\
&\?\times \bigg(\sum_{m\in\BN}\binom{-m}{i}r_{\alpha,m}x^{-m-i}\bigg)\nonumber\\
&=\sum_{k,m\in\BN}\binom{\alpha+k-m}{n}s_{\alpha,k}E^-(\alpha,x)E^+(\alpha,x)r_{\alpha,m}x^{k-m-n}
\end{align}
on $V_{L,\BF}$, where $\binom{\alpha+k-m}{n}(u\otimes\iota(e_{\gamma}))
=\binom{\langle\alpha,\gamma\rangle+k-m}{n}u\otimes\iota(e_{\gamma})$ for $u\in M(1)_{\BF}$ and $\gamma\in L$.
Note that $M(1)_{\BF}$ is a vertex subalgebra of $V_{L,\BF}$.
Then we obtain the following result.
\begin{proposition}\label{prop:Mu-Heisen-gen}
The vertex algebra $M(1)_{\BF}$ is generated by 
$\{s_{\gamma^{i},p^{k}}\mid 1\le i\le d,k\in \BN\}$.
\end{proposition}

\begin{proof}
We denote by $U$ the subalgebra of $M(1)_{\BF}$ generated 
from $\mathbf{1}$ by $\{s_{\gamma^i,p^k}\mid 1\le i\le d, k\in \BN\}$.
Let $\alpha\in \{\gamma^1,\ldots,\gamma^d\}$.
We first use induction on $n$ to show that the following two inclusions hold:
\begin{align}
\bigg(\sum_{k\in \BN}\binom{k}{n}s_{\alpha,k}x^{k-n}\bigg)E^{-}(\alpha,x)U&\subset U((x)),\label{eq:Mu-9}\\
E^{+}(\alpha,x)\bigg(\sum_{m\in\BN}\binom{-m}{n}r_{\alpha,m}x^{-m-n}\bigg)U&\subset U((x)).\label{eq:Mu-10}
\end{align}
Note that if \eqref{eq:Mu-9} holds for some $n$, then
\begin{equation}\label{eq:Mu-11}
s_{\alpha,n}U\subset U,
\end{equation}
which implies
\begin{equation}\label{eq:Mu-12}
s_{\alpha,n}\in U,
\end{equation}
since $\mathbf{1}\in U$.

Setting $n=1$ in \eqref{eq:Y-s-alpha-n}, we have
\begin{equation*}
    Y(s_{\alpha,1},x)=\bigg(\sum_{k\ge0} \binom{k}{1}s_{\alpha,k}x^{k-1}\bigg)E^-(\alpha,x)
    +E^+(\alpha,x)\bigg(\sum_{m\ge0} \binom{-m}{1}r_{\alpha,m}x^{-m-1}\bigg).
\end{equation*}
Since $Y(s_{\alpha,1},x)U\subset U((x))$, we see that \eqref{eq:Mu-9} 
and \eqref{eq:Mu-10} hold for $n=1$.
Suppose that \eqref{eq:Mu-9} and \eqref{eq:Mu-10} hold for $1\le n <p^{t}$.
We use induction on $n$ to show that \eqref{eq:Mu-9} and \eqref{eq:Mu-10} 
hold for $p^t\le n< p^{t+1}$.
Since $Y(s_{\alpha,p^t},x)U\subset U((x))$,
it follows from \eqref{eq:Y-s-alpha-n} and the induction hypothesis 
that \eqref{eq:Mu-9} and \eqref{eq:Mu-10} hold for $n=p^t$.
A direct calculation shows that
\begin{equation}\label{eq:Mu-13}
\bigg(\sum_{k\ge p^t}\binom{k}{p^t}s_{\alpha,k}x^{k-p^t}\bigg)E^{-}(\alpha,x)
=\sum_{m\ge p^t}\sum_{i=0}^{m-p^t}\binom{m-i}{p^t}s_{\alpha,m-i}s_{-\alpha,i}x^{m-p^t}.
\end{equation}
Suppose that \eqref{eq:Mu-9} and \eqref{eq:Mu-10} hold for $0\le n<m$, 
where $p^t\le m <p^{t+1}$.
Then $s_{-\alpha,n}U\subset U$ for $0\le n<m$ by Lemma \ref{lem:DG-s--alpha}.
By \eqref{eq:Mu-13}, we have
\begin{equation}
\sum_{i=0}^{m-p^t}\binom{m-i}{p^t}s_{\alpha,m-i}s_{-\alpha,i}U\subset U.
\end{equation}
Consequently, $\binom{m}{p^t}s_{\alpha,m}U\subset U$ by the induction hypothesis.
Since $p^t\le m<p^{t+1}$, we have $\binom{m}{p^t}\not\equiv 0\pmod{p}$, 
hence \eqref{eq:Mu-11} holds for $n=m$.
Then \eqref{eq:Mu-12} holds for $n=m$, so that $Y(s_{\alpha,m},x)U\subset U((x))$.
By \eqref{eq:Y-s-alpha-n} and the induction hypothesis, 
we see that \eqref{eq:Mu-9} and \eqref{eq:Mu-10} hold for $n=m$,
completing our induction.

Now for any $\alpha\in \{\gamma^1,\ldots,\gamma^d\}$ and $n\in \BN$, 
we see that $s_{\alpha,n}\in U$.
By induction on $k$, it follows from \eqref{eq:Mu-11} that 
$s_{\alpha^1,n_1}\cdots s_{\alpha^k,n_k}\in U$ for 
$\alpha^j\in \{\gamma^1,\ldots,\gamma^d\}$ and $n_j,k\in \BN$.
Thus, $M(1)_{\BF}=U$, as desired.
\end{proof}

Let $G_{L}=(\langle\gamma^i,\gamma^j\rangle)_{d\times d}$ be the Gram matrix of $L$.
For later use, we recall the following results.
\begin{theorem}[{\cite[Theorem 13]{Mu}}]\label{th:Mu-Theor}
	  If $\det G_{L}\not \equiv 0 \pmod{p}$, then $V_{L,\BF}$ is a simple vertex algebra.
\end{theorem}
\begin{theorem}[{\cite[Theorem 3.8]{ZM}}]\label{th:ZM}
Let $\{\beta_{(1)},\ldots,\beta_{(l)}\}$ be a complete set of coset 
representatives for $L^*/L$.  
If $\det G_{L}\not\equiv 0 \pmod{p}$, then the modules $V_{\beta_{(i)}+L,\BF}$,
$1\le i\le l$, are pairwise inequivalent irreducible $V_{L,\BF}$-modules.
\end{theorem}

\section{Lattice vertex algebras $V_{L,\BF}$ of type ADE with $\det G_{L}\not\equiv 0\pmod{p}$}\label{sec:3}

In this section, let $L$ be a root lattice of type ADE with 
$\det G_{L}\not\equiv 0\pmod{p}$.
We establish an isomorphism between the lattice vertex algebra 
$V_{L,\BF}$ and the simple affine vertex algebra of the same 
type at level one over $\BF$.
Using this isomorphism, we then classify the irreducible 
$\BN$-graded modules of $V_{L,\BF}$, viewed as an $\BN$-graded vertex algebra.

\subsection{ADE-type lattice and affine vertex algebras: an isomorphism}

Let $\fg$ be a Lie algebra of type ADE over $\BF$.
Fix a Cartan subalgebra $\fh \subset \fg$.
Let $\Delta$ be the corresponding root system, 
let $\Pi=\{\alpha_1,\ldots,\alpha_n\}$ be the set of simple roots, 
and let $L$ be the root lattice.
For the remainder of this paper, we assume that $p>2$, and that
$p>3$ if $\fg$ is of type $E_6$.

Under these assumptions, the Lie algebra $\fg$ of type 
$D_n$ $(n\ge 4)$, $E_6$, $E_7$, or $E_8$ over $\BF$ is simple (see, for example, \cite[\S 0.13]{H1995}).
For type $A_n$ $(n\ge 1)$, if $p \nmid n+1$, 
then $\fg=\mathfrak{sl}(n+1, \BF)$ remains simple (see \cite[\S 0.13]{H1995}). 
Otherwise, the Lie algebra $\mathfrak{sl}(n+1, \BF)$ 
has the one-dimensional center $\BF I_{n+1}$, 
which is a restricted ideal;
here $I_{n+1}$ denotes the identity matrix of size $n+1$,
and the quotient
\begin{equation}\label{eq:psl}
\mathfrak{psl}(n+1,\BF) := \mathfrak{sl}(n+1, \BF) / \BF I_{n+1}
\end{equation}
is simple (see \cite[\S 4.1]{SF}).

Recall from \cite[\S 5.6]{Kac2} or \cite{FLM} that 
there exist $E_{\alpha}\in \fg$ ($\alpha\in \Delta$) such that 
\begin{equation}
\fg=\fh\oplus(\bigoplus_{\alpha\in \Delta}\BF E_{\alpha}),
\end{equation}
and the following relations hold:
\begin{equation}\label{eq:3.2}
[\fh,\fh]=0,\;\;[h,E_{\alpha}]=\langle h,\alpha\rangle E_{\alpha},\;\;
[E_{\alpha},E_{\beta}]=\begin{cases}
\epsilon(\alpha,-\alpha)\alpha&\text{if }\alpha+\beta=0\\
\epsilon(\alpha,\beta)E_{\alpha+\beta}&\text{if }\alpha+\beta\in \Delta\\
0&\text{if }\alpha+\beta\notin \Delta\cup\{0\}
\end{cases}
\end{equation}
for all $h\in\fh$ and $\alpha,\beta\in \Delta$.
Here, $\epsilon(\cdot, \cdot)=(-1)^{\epsilon_{0}(\cdot, \cdot)}$, 
with $\epsilon_0(\cdot, \cdot)$ being the $2$-cocycle of the 
central extension of the root lattice $L$ (see, e.g., \cite[\S 7.1]{FLM}).
The invariant bilinear form $\langle\cdot,\cdot\rangle$ on $\fg$ is given by
\begin{equation}\label{eq:3.3}
\langle \fh,E_{\alpha}\rangle=0,\;\;\langle E_{\alpha}, E_{\beta}\rangle=
\begin{cases}
\epsilon(\alpha,-\alpha) &\text{if }\alpha+\beta=0\\
0&\text{if }\alpha+\beta\ne 0
\end{cases}
\end{equation}
for $\alpha,\beta\in \Delta$.
For $\alpha_i\in \Pi$, set $h_{\alpha_i}=\frac{2\alpha_i}{\langle\alpha_i,\alpha_i\rangle}=\alpha_i\in \fh(\cong \fh^*)$.
Then the set
\begin{equation*}
\{h_{\alpha_i}, E_{\alpha}\mid \alpha_{i}\in \Pi, \alpha\in \Delta\}
\end{equation*}
is a (Chevalley) basis of $\fg$.

Consider the affine Lie algebra
\begin{equation*}
\hat{\fg}=\fg\otimes \BF[t,t^{-1}]\oplus\BF\bk,
\end{equation*}
where $\bk$ is central, 
and the Lie bracket is given by
\begin{equation}\label{eq:Lie-brac}
[a\otimes t^{r},b\otimes t^{s}]=[a,b]\otimes t^{r+s}+r\langle a,b\rangle\delta_{r+s,0}\bk
\end{equation}
for $a,b\in \fg$ and $r,s\in \BZ$.

For any $\hat{\fg}$-module $W$ and $a\in \fg$, 
define
\begin{equation}
a(x)=\sum_{m\in \BZ}a(m)x^{-m-1}\in (\End W)[[x,x^{-1}]],
\end{equation}
where $a(m)$ denotes the action of $a\otimes t^m$.
If $\bk$ acts as a scalar $\ell\in \BF$ on a $\hat{\fg}$-module $W$, 
we say that $W$ is of \emph{level $\ell$}.

Let $\ell\in \BF$. 
Denote by $\BF_{\ell}$ the $1$-dimensional ($\fg[t]\oplus\BF\bk$)-module $\BF$, 
where $\fg[t]$ acts trivially and $\bk$ acts as the scalar $\ell$. 
Then one can form an induced module
\begin{equation}
V_{\hat{\fg}}(\ell,0):=U(\hat{\fg})\otimes_{U(\fg[t]\oplus\BF\bk)}\BF_{\ell},
\end{equation}
which is a $\hat{\fg}$-module of level $\ell$.
It follows from \cite{JLM} (cf. \cite{FZ}) that 
$(V_{\hat{\fg}}(\ell,0),Y,\mathbf{1})$ is a vertex algebra over $\BF$, 
which is uniquely determined by the conditions that $\mathbf{1}:=1\otimes 1$  and 
\begin{equation*}
Y(a,x)=a(x)\in (\End V_{\hat{\fg}}(\ell,0))[[x,x^{-1}]]\text{ for } a\in \fg. 
\end{equation*}
Recall from \cite[Proposition 5.2]{JLM} that $V_{\hat{\fg}}(\ell,0)$ 
is an $\BN$-graded vertex algebra.
Moreover, $V_{\hat{\fg}}(\ell,0)$ possesses a unique maximal 
graded ideal (see \cite[Lemma 2.7]{JLM}), denoted by $J$. 
Consequently, the quotient
\begin{equation}
L_{\hat{\fg}}(\ell,0):=V_{\hat{\fg}}(\ell,0)/J
\end{equation}
is a simple $\BN$-graded vertex algebra with $L_{\hat{\fg}}(\ell,0)_{(0)}=\BF\mathbf{1}$. 

Now we consider the simple affine vertex algebra $L_{\hat{\fg}}(1,0)$.
For $\alpha_i,\alpha_j\in \Pi$ and $\alpha, \beta\in \Delta$, 
the equation \eqref{eq:Lie-brac} is equivalent to 
\begin{align}
[h_{\alpha_{i}}(x_1),h_{\alpha_j}(x_2)]&=-\langle h_{\alpha_i},h_{\alpha_j}\rangle \partial^{(1)}_{x_1}x_2^{-1}\delta\left(\frac{x_1}{x_2} \right),\label{eq:h-h}\\
[h_{\alpha_{i}}(x_1),E_{\alpha}(x_2)]&=\langle h_{\alpha_i},\alpha\rangle E_{\alpha}(x_2)x_2^{-1}\delta\left( \frac{x_1}{x_2}\right), \label{eq:h-E}\\
[E_{\alpha}(x_1),E_{\beta}(x_2)]&=\sum_{r\in \BZ}\sum_{s \in \BZ}[E_{\alpha}(r),E_{\beta}(s)]x_1^{-r-1}x_2^{-s-1}\nonumber\\
&=\sum_{r\in \BZ}\sum_{s \in \BZ}\big( [E_{\alpha},E_{\beta}](r+s)+r\langle E_{\alpha},E_{\beta} \rangle\delta_{r+s,0} \big) x_1^{-r-1}x_2^{-s-1}.\nonumber
\end{align}
Let $\alpha,\beta\in \Delta$.
Using \eqref{eq:3.2}--\eqref{eq:3.3}, we treat the cases separately.
\begin{enumerate}[(i)]
\item If $\alpha+\beta\notin \Delta\cup \{0\}$, then 
\begin{equation}\label{eq:1}
[E_{\alpha}(x_1),E_{\beta}(x_2)]=0.
\end{equation}

\item If $\alpha+\beta=0$, then
\begin{align}\label{eq:2}
[E_{\alpha}(x_1),E_{\beta}(x_2)]&=\sum_{r\in \BZ}\sum_{s\in \BZ}\epsilon(\alpha,-\alpha)\alpha(r+s)x_1^{-r-1}x_2^{-s-1}
+\sum_{r\in \BZ}r\epsilon(\alpha,-\alpha)x_1^{-r-1}x_2^{r-1}\nonumber\\
&=\epsilon(\alpha,-\alpha)\biggl(\alpha(x_2) x_2^{-1}\delta\biggl(\frac{x_1}{x_2}\biggr)-\partial_{x_1}^{(1)}x_2^{-1}\delta\biggl(\frac{x_1}{x_2}\biggr)\biggr).
\end{align}

\item If $\alpha+\beta\in \Delta$, then 
\begin{align}\label{eq:3}
[E_{\alpha}(x_1),E_{\beta}(x_2)]&=\sum_{r\in \BZ}\sum_{s\in \BZ}\epsilon(\alpha,\beta)E_{\alpha+\beta}(r+s)x_1^{-r-1}x_2^{-s-1}\nonumber\\
&=\epsilon(\alpha,\beta) E_{\alpha+\beta}(x_2)x_2^{-1}\delta\bigg(\frac{x_1}{x_2}\bigg).
\end{align}
\end{enumerate}

The following is an analogue over fields of prime characteristic of \cite[Theorem 5.6(c)]{Kac2}.

\begin{theorem}\label{th:iso-1}
Let $\fg$ be a Lie algebra of type $A_n$ $(n\ge 1\text{ and }p\nmid n+1)$, 
$D_n$ $(n\ge 4)$, $E_{6}$ $(p>3)$, $E_7$, or $E_8$ over $\BF$, 
and let $L$ be the root lattice of $\fg$.
Then the vertex algebra $V_{L,\BF}$ is isomorphic to $L_{\hat{\fg}}(1,0)$.
\end{theorem}

\begin{proof}
Let $\alpha_i,\alpha_j\in \Pi$.
Then (see, for example, \cite[(6.5.3)]{LL})
\begin{align}\label{eq:s-alpha-i-j}
[Y(s_{\alpha_{i},1},x_1),Y(s_{\alpha_j,1},x_2)]&=[Y(\alpha_i(-1)\mathbf{1},x_1), Y(\alpha_{j}(-1)\mathbf{1},x_2)]\nonumber\\
&=[\alpha_{i}(x_1),\alpha_{j}(x_2)]\nonumber\\
&=-\langle \alpha_i,\alpha_j\rangle \partial^{(1)}_{x_1}x_2^{-1}\delta\biggl(\frac{x_1}{x_2} \biggr).
\end{align}
For $\alpha\in \Delta$, we have (see \cite[(6.5.9)]{LL})
\begin{align}\label{eq:s-alpha-i-iota}
[Y(s_{\alpha_i,1},x_1),Y(\iota(e_{\alpha}),x_2)]&=[Y(\alpha_i(-1)\mathbf{1},x_1),Y(\iota(e_{\alpha}),x_2)]\nonumber\\
&=\langle\alpha_i,\alpha\rangle Y(\iota(e_{\alpha}),x_2)x_2^{-1}\delta\biggl(\frac{x_1}{x_2}\biggr).
\end{align}
Let $\alpha, \beta\in \Delta$. 
Using the commutator formula (see \cite[(3.1.8)]{LL}), 
we obtain
\begin{align}\label{eq:comm}
&\?[Y(\iota(e_{\alpha}),x_1), Y(\iota(e_{\beta}),x_2)]\nonumber\\
&=\operatorname{Res}_{x_0}x_2^{-1}\delta\biggl(\frac{x_1-x_0}{x_2}\biggr)Y(Y(\iota(e_{\alpha}),x_0)\iota(e_{\beta}),x_2)\nonumber\\
&=\operatorname{Res}_{x_0}\sum_{k\in\BZ}\sum_{m\in \BZ}(x_1-x_0)^{m}x_2^{-m-1}Y(\iota(e_{\alpha})_k\iota(e_{\beta}),x_2)x_0^{-k-1}\nonumber\\
&=\sum_{i\in \BN}(-1)^{i}Y(\iota(e_{\alpha})_{i}\iota(e_{\beta}),x_2)\partial_{x_1}^{(i)}x_2^{-1}\delta\biggl(\frac{x_1}{x_2}\biggr).
\end{align}
By the definition of $Y(\iota(e_{\alpha}),x)$ in\eqref{eq:ver-ope}
and by \cite[(6.4.37), (6.4.51)]{LL}, we have
\begin{equation*}
Y(Y(\iota(e_{\alpha}),x_0)\iota(e_{\beta}),x_2)
=\epsilon(\alpha,\beta)x_0^{\langle\alpha,\beta\rangle}Y(E^{-}(-\alpha,x_0)\iota(e_{\alpha+\beta}),x_2).
\end{equation*}
It follows that for all $m \in \BZ$,
\begin{equation}\label{eq:3.9}
Y(\iota(e_{\alpha})_{m}\iota(e_{\beta}),x_2)=\epsilon(\alpha,\beta)Y(s_{\alpha,-m-1-\langle\alpha,\beta\rangle}\iota(e_{\alpha+\beta}),x_2).
\end{equation}
We now proceed case by case.
\begin{enumerate}[(i)]
\item If $\alpha+\beta\not\in\Delta\cup \{0\}$, 
then $\langle\alpha,\beta\rangle=0$, $1$, or $2$. 
Since $s_{\alpha,m}=0$ for $m<0$,
it follows from \eqref{eq:3.9} that $Y(\iota(e_{\alpha})_{k}\iota(e_{\beta}),x_2)=0$ for $k\in\BN$.
Then \eqref{eq:comm} yields
\begin{equation}\label{eq:1'}
[Y(\iota(e_{\alpha}),x_1), Y(\iota(e_{\beta}),x_2)]=0.
\end{equation}

\item If $\alpha+\beta=0$, then $\langle\alpha,\beta\rangle=-\langle\alpha,\alpha\rangle=-2$. 
It follows from \eqref{eq:comm} and \eqref{eq:3.9} that
\begin{equation}\label{eq:2'}
[Y(\iota(e_{\alpha}),x_1), Y(\iota(e_{\beta}),x_2)]=\epsilon(\alpha,-\alpha)\biggl( Y(\alpha(-1)\mathbf{1},x_2)x_2^{-1}\delta\biggl(\frac{x_1}{x_2}\biggr)-  \partial_{x_1}^{(1)}x_2^{-1}\delta\biggl(\frac{x_1}{x_2}\biggr)\biggr).
\end{equation}

\item If $\alpha+\beta\in \Delta$, then $\langle\alpha,\beta\rangle=-1$. 
It follows from \eqref{eq:comm} and \eqref{eq:3.9} that
\begin{equation}\label{eq:3'}
[Y(\iota(e_{\alpha}),x_1),  Y(\iota(e_{\beta}),x_2)]=\epsilon(\alpha,\beta)Y(\iota(e_{\alpha+\beta}),x_2)x_2^{-1}\delta\biggl(\frac{x_1}{x_2}\biggr).
\end{equation}
\end{enumerate}

Since $h_{\alpha_i}=\alpha_i$ for all $\alpha_i\in \Pi$, 
a comparison of \eqref{eq:h-h}--\eqref{eq:3} with \eqref{eq:s-alpha-i-j}--\eqref{eq:s-alpha-i-iota} and \eqref{eq:1'}--\eqref{eq:3'} shows that $V_{L,\BF}$ is a $\hat{\fg}$-module with the actions
\begin{equation}\label{eq:3.14}
E_{\alpha}(x)=Y(\iota(e_{\alpha}),x),\;\;h_{\alpha_i}(x)=Y(s_{\alpha_i,1},x)
\end{equation}
for all $\alpha\in \Delta$ and $\alpha_i\in \Pi$.
Then $V_{L,\BF}$ is a $\hat{\fg}$-module of level $1$ generated by $\mathbf{1}$.
Moreover, by \eqref{eq:3.14},  we have
$E_{\alpha}(m)\mathbf{1}=0$ and $h_{\alpha_i}(m)\mathbf{1}=0$ for 
$\alpha\in \Delta$, $\alpha_i\in \Pi$, and $m\in \BN$.
It follows from the universal property of 
$V_{\hat{\fg}}(1,0)$ (see \cite[Remark 6.2.8]{LL}) that
there exists a $\hat{\fg}$-module homomorphism 
$\psi:V_{\hat{\fg}}(1,0)\to  V_{L,\BF}$ such that
\begin{equation}
\psi(\mathbf{1})=\mathbf{1}.
\end{equation}
Then we obtain
\begin{equation}\label{eq:3.23}
\psi \big(E_{\gamma_1}(m_1)\cdots E_{\gamma_k}(m_k)\mathbf{1}\big)=\iota(e_{\gamma_1})_{m_1}\cdots\iota(e_{\gamma_k})_{m_k} \mathbf{1}
\end{equation}
for $m_1,\ldots,m_k\in\BZ$, $\gamma_1,\ldots,\gamma_k\in \pm \Pi(:=\Pi \cup (-\Pi))$, and $k\in \BN$. 
Consequently, for $\alpha\in \pm \Pi$, we have
\begin{align*}
\psi (Y(E_{\alpha}(-1)\mathbf{1},x)E_{\gamma_1}(m_1)\cdots E_{\gamma_k}(m_k)\mathbf{1})&=
\sum_{m\in\BZ}\psi\big(E_{\alpha}(m)E_{\gamma_1}(m_1)\cdots E_{\gamma_k}(m_k)\mathbf{1} \big)x^{-m-1} \\
&=\sum_{m\in\BZ}\big(\iota(e_{\alpha})_{m}\iota(e_{\gamma_1})_{m_1}\cdots\iota(e_{\gamma_k})_{m_k} \mathbf{1}\big) x^{-m-1}\\
&=Y(\iota(e_{\alpha}),x) \psi(E_{\gamma_1}(m_1)\cdots E_{\gamma_k}(m_k)\mathbf{1}).
\end{align*}
Since the vertex algebra $V_{\hat{\fg}}(1,0)$ is generated 
by  $\{E_{\alpha}(-1)\mathbf{1}\mid \alpha\in \pm \Pi\}$,
it follows from \cite[Proposition 5.7.9]{LL} that 
$\psi$ is a vertex algebra homomorphism.

Equation \eqref{eq:3.23} shows that $\psi$ is surjective, 
so $V_{\hat{\fg}}(1,0)/\ker \psi$ is isomorphic to $V_{L,\BF}$ as vertex algebras.
Recall that  $L_{\hat{\fg}}(1,0)=V_{\hat{\fg}}(1,0)/J$ is 
the simple quotient, where $J$ is the maximal graded ideal.
Since $\psi$ is a grading-preserving vertex algebra homomorphism, 
we deduce that $\ker\psi$ is a graded ideal, and hence $\ker\psi\subset J$.
Furthermore, by our assumptions on $p$ (namely, $p>2$, 
with $p\nmid n+1$ if $L$ is of type $A_n$, and $p>3$ if $L$ is of type $E_6$), 
we have $\det G_{L}\not\equiv 0\pmod{p}$.
Theorem \ref{th:Mu-Theor} then shows that $V_{L,\BF}$ is simple.
Consequently, $V_{L,\BF}$ is isomorphic to $L_{\hat{\fg}}(1,0)$ as 
vertex algebras and $J=\ker \psi$.
\end{proof}

\subsection{The classification of irreducible $\BN$-graded modules}\label{sec:3.2}

We begin by recalling the Zhu algebra $A(V)$ associated to an 
$\BN$-graded vertex algebra $V$ (see \cite{DR}, \cite{FZ}, \cite{Zhu1996}).
Following \cite{Zhu1996}, for $u,v\in V$ with $u$ $\BZ$-homogeneous, 
define
\begin{equation}\label{eq:u*v}
u*v=\Res_{x} \frac{(1+x)^{\deg u}}{x} Y(u,x)v.
\end{equation}
Then $(V,*)$ carries a nonassociative algebra structure.
On the other hand, following \cite{DR}, 
for $u,v\in V$ with $u$ $\BZ$-homogeneous and for $k\in \BN$, set
\begin{equation*}
u\circ_{k}v=\Res_{x} \frac{(1+x)^{\deg u}}{x^{k+2}} Y(u,x)v,
\end{equation*}
and set $O(V):=\spanf\{u\circ_{k} v\mid u,v\in V, k\in \BN\}$.

\begin{proposition}[\cite{DR}]
The subspace $O(V)$ is an ideal of the nonassociative algebra $(V,*)$, 
and the quotient $A(V):=V/O(V)$ is an associative algebra, 
where $\mathbf{1}+O(V)$ is the identity.
\end{proposition}

Recall that $\{E_{\alpha},h_{\alpha_i}\mid \alpha\in \Delta, 
\alpha_i\in \Pi\}$ is the Chevalley basis of $\fg$.
From \eqref{eq:3.2}, we obtain
\begin{equation}\label{eq:ad}
(\ad E_{\alpha})^{3}=0,\;\;(\ad h_{\alpha_i})^{p}=\ad h_{\alpha_i}\quad 
\text{for }\alpha\in \Delta \text{ and } \alpha_i\in\Pi.
\end{equation}
It is well known that classical Lie algebras carry the structure of 
a restricted Lie algebra (see \cite{SF}).
Since $p\ge 3$, by \eqref{eq:ad} and \cite[Theorem 2.3 in \S 2]{SF}, 
we fix the $p$-mapping of $\fg$ as follows:
\begin{equation}\label{eq:4.30}
E_{\alpha}^{[p]}=0,\;\;h_{\alpha_i}^{[p]}=h_{\alpha_i} \quad\text{for}\; \alpha\in \Delta\text{ and }\alpha_i\in \Pi.
\end{equation}
Let $J_{0}$ be the $\hat{\fg}$-submodule of $V_{\hat{\fg}}(\ell,0)$ 
generated by vectors
\begin{equation}
(a(-m)^{p}-a^{[p]}(-pm))\mathbf{1}
\end{equation}
for $a\in \fg$ and $m\in \BZ_+$. 
By \cite[Proposition 5.6]{JLM}, 
$J_{0}$ is an ideal of the vertex algebra $V_{\hat{\fg}}(\ell,0)$ and 
coincides with the ideal generated by $(a(-1)^{p}-a^{[p]}(-p))\mathbf{1}$ for $a\in \fg$.
Then
\begin{equation*}
V_{\hat{\fg}}^{0}(\ell,0):=V_{\hat{\fg}}(\ell,0)/J_{0}
\end{equation*}
is an $\BN$-graded vertex algebra (see \cite[(5.11)]{JLM}).
The following result was proved in \cite[Proposition 6.1, Theorem 6.3]{JLM}.

\begin{theorem}\label{th:JLM}
For $v\in V_{\hat{\fg}}(\ell,0)$, 
set $[v]:=v+O(V_{\hat{\fg}}(\ell,0))\in A(V_{\hat{\fg}}(\ell,0))$.
The linear map 
\begin{align*}
\phi: \fg&\to A(V_{\hat{\fg}}(\ell,0))\\
a&\mapsto  [a]
\end{align*}
extends uniquely to an algebra isomorphism from $U(\fg)$ to 
$A(V_{\hat{\fg}}(\ell,0))$.
Furthermore, let $\mathfrak{u}(\fg)=U(\fg)/\langle a^p-a^{[p]} \mid a\in \fg\rangle$ 
be the restricted enveloping algebra of $\fg$. 
This algebra isomorphism
descends to an isomorphism from $\mathfrak{u}(\fg)$ to $A(V_{\hat{\fg}}^{0}(\ell,0))$.
\end{theorem}

By \eqref{eq:4.30}, the restricted enveloping algebra $\mathfrak{u}(\fg)$ is the quotient
\begin{equation}
\mathfrak{u}(\fg)= U(\fg)/\langle E_{\alpha}^{p},h_{\alpha_{i}}^{p}-h_{\alpha_{i}}\mid \alpha\in \Delta, \alpha_i\in \Pi\rangle.
\end{equation}
Then we recall the following result.
\begin{theorem}[\cite{C}]\label{th:the-class-u(g)-mod}
The highest weight modules $L(\mu)$ exhaust all irreducible $\mathfrak{u}(\fg)$-modules, 
up to equivalence, where  $\mu=\mu_1\omega_1+\cdots+\mu_n\omega_n$ with $0\le \mu_i\le p-1$, 
and $\omega_i$ $(1\le i\le n)$ are the fundamental weights, 
i.e., $\langle\omega_i,\alpha_j^{\vee}\rangle=\delta_{i,j}$ 
for all simple coroots $\alpha_j^{\vee}$. 
\end{theorem}

On the other hand, for $\alpha\in\Delta$, we have
\begin{align*}
    \iota(e_\alpha)_{-1}\iota(e_\alpha)=\Res_x x^{-1}Y(\iota(e_\alpha),x)\iota(e_\alpha)=\Res_x \epsilon(\alpha,\alpha)xE^-(-\alpha,x)\iota(e_{2\alpha})=0.
\end{align*}
The isomorphism of vertex algebras given in Theorem \ref{th:iso-1} 
implies $\psi(E_{\alpha}(-1)^{2}\mathbf{1})=\iota(e_\alpha)_{-1}\iota(e_\alpha)$.
Thus we have the following result, which generalizes \cite[Lemma 6.9]{JLM}.
\begin{lemma}\label{lem:E2=0}
For all  $\alpha\in \Delta$,  $E_{\alpha}(-1)^{2}\mathbf{1}=0$ in $L_{\hat{\fg}}(1,0)$.
\end{lemma}

The main result of this section is as follows.

\begin{theorem}\label{th:clas-type-ADE}
Let $L$ be the root lattice of one of the following types: 
$A_n$ $(n\ge 1\text{ and }p\nmid n+1)$, 
$D_n$ $(n\ge 4)$, $E_{6}$ $(p>3)$, $E_7$, or $E_8$.
Then the modules $V_{\beta+L,\BF}$, where $\beta$ runs over a complete set 
of representatives of $L^*/L$, exhaust all irreducible $\BN$-graded $V_{L,\BF}$-modules
up to equivalence.
\end{theorem}

\begin{proof}
Note that $L_{\hat{\fg}}(1,0)$ is a quotient of $V_{\hat{\fg}}^{0}(1,0)$.
Then $A(L_{\hat{\fg}}(1,0))$ is a quotient algebra of 
$A(V_{\hat{\fg}}^{0}(1,0))$ (see \cite[Theorem 3.3]{DR}). 
By Lemma \ref{lem:E2=0}, we also have $[E_{\alpha}]^2=0$ in 
$A(L_{\hat{\fg}}(1,0))$ for all $\alpha\in \Delta$. 
Consequently, Theorem \ref{th:JLM} shows that $A(L_{\hat{\fg}}(1,0))$ 
is isomorphic to a quotient of $\mathfrak{u}(\fg)/\langle (E_{\alpha})^2\mid \alpha\in \Delta\rangle$.

We now consider the irreducible modules of 
$\mathfrak{u}(\fg)/\langle (E_{\alpha})^2\mid \alpha\in \Delta\rangle$.
Let $L(\mu)$ be an irreducible module of $\mathfrak{u}(\fg)$.
By Theorem \ref{th:the-class-u(g)-mod}, $L(\mu)$ is a highest weight 
module with $\mu=\sum_{i=1}^{n}\mu_i\omega_i$ and $0\le \mu_i\le p-1$ for all $i$.
Let $\Delta_+$ be the set of positive roots.
We claim that such an $L(\mu)$ gives an irreducible module for
$\mathfrak{u}(\fg)/\langle (E_{\alpha})^2\mid \alpha\in \Delta\rangle$ 
if and only if
\begin{equation}\label{eq:hig-wei-res}
	\mu (h_{\alpha})=\langle \mu,\alpha^{\vee}\rangle=\sum_{i=1}^{n}\mu_i\langle \omega_i,\alpha^{\vee}\rangle = 0\text{ or }1
\end{equation}
for all $\alpha\in \Delta_+$.
The ``if'' part is straightforward to check. 
For the ``only if'' part, 
if \eqref{eq:hig-wei-res} fails, then there exists some 
$\alpha\in \Delta_+$ such that $(E_{-\alpha})^2 w\neq 0$, 
where $w$ is a highest weight vector of $L(\mu)$, 
contradicting the requirement that $(E_{-\alpha})^2$ acts trivially.

Let $M_n$ be the Cartan matrix of $\fg$.
The highest weights of all irreducible modules $L(\mu)$ for 
$\mathfrak{u}(\fg)/\langle (E_{\alpha})^2\mid \alpha\in \Delta\rangle$ 
are listed in Table \ref{tab:ADE-data}.
From Table \ref{tab:ADE-data}, the numbers of irreducible highest 
weight modules $L(\mu)$ are $n+1$, $4$, $3$, $2$, and $1$ for $\fg$ of 
type $A_n$ ($n\ge 1$), $D_n$ ($n\ge 4$), $E_6$, $E_7$, and $E_8$, respectively.
These numbers equal $\det M_n$ in each case.
Note that the Gram matrix $G_{L}$ of the ADE-type root lattice 
$L$ is exactly the Cartan matrix of the same type.
Hence $A(L_{\hat{\fg}}(1,0))$ has at most $\det G_L$ 
irreducible modules up to equivalence.
Consequently, by \cite[Propositions 2.10, 2.11]{JLM}, 
$L_{\hat{\fg}}(1,0)$ has at most $\det G_L$ irreducible 
$\BN$-graded modules up to equivalence.

\begin{table}[htbp]
\centering
\caption{The highest weights $\mu$}
\label{tab:ADE-data}
\begin{tabular}{lccc}
\hline
type&   the highest root $\theta$ & $\mu$ & $\det M_n$ \\
\hline
$A_n$&$\alpha_1+\cdots+\alpha_n$ & $0$, $\omega_1$, $\ldots$, $\omega_n$ & $n+1$    \\
$D_n$ & $\alpha_1+2\alpha_2+\cdots+2\alpha_{n-2}+\alpha_{n-1}+\alpha_n$ & $0$, $\omega_1$, $\omega_{n-1}$, $\omega_n$& $4$ \\
$E_6$ & $\alpha_1+2\alpha_2+2\alpha_3+3\alpha_4+2\alpha_5+\alpha_6$ & $0$, $\omega_1$, $\omega_6$ & $3$\\
$E_7$& $2\alpha_1+2\alpha_2+3\alpha_3+4\alpha_4+3\alpha_5+2\alpha_6+\alpha_7$& $0$, $\omega_7$ & $2$\\
$E_8$& $2\alpha_1+3\alpha_2+4\alpha_3+6\alpha_4+5\alpha_5+4\alpha_6+3\alpha_7+2\alpha_8$ & $0$ & $1$\\
\hline
\end{tabular}
\end{table}

On the other hand, recall that $V_{L,\BF}$ is an $\BN$-graded vertex algebra.
For $\beta\in L^*$, $V_{\beta+L,\BF}$ can be regarded as an $\BN$-graded 
module for $V_{L,\BF}$ with $V_{\beta+L,\BF}=\bigoplus_{m\in \BN}(V_{\beta+L,\BF})_{(m)}$,
where $(V_{\beta+L,\BF})_{(m)}$ is the linear span of elements of the form 
$s_{\alpha^1,m_1}\cdots s_{\alpha^k,m_k}\iota(e_{\beta+\alpha})$ 
with $k\in \BN$, $m_i\in \BZ_+$, $\alpha^i\in \Pi$, and $\alpha\in L$ such that 
$\frac{\langle\beta+\alpha,\beta+\alpha\rangle}{2}-\mathop{\min}\limits_{\gamma\in\beta+ L}\frac{\langle\gamma,\gamma\rangle}{2}+\sum_{i=1}^{k}m_i=m$ $(\in \BN)$.
By the assumption on $p$, we have $\det G_{L}\not\equiv 0\pmod{p}$.
It follows from Theorem \ref{th:ZM} that there exist $\det G_{L}$ 
inequivalent irreducible $\BN$-graded modules of $V_{L,\BF}$, 
since $|L^*/L|=\det G_{L}$ (see, for example, \cite[Section 2.3]{Griess}).
Therefore, the desired result follows from Theorem \ref{th:iso-1}.
\end{proof}

\section{Lattice vertex algebras $V_{L,\BF}$ of type $A_n$ with $n+1 \equiv 0\pmod{p}$}\label{sec:4}

In this section, let $L$ be a root lattice of type $A_n$ with 
$n+1\equiv 0\pmod{p}$.
We determine the nonzero maximal $\BN$-graded ideal of $V_{L,\BF}$, 
and obtain certain irreducible modules for $V_{L,\BF}$.
Furthermore, when $n+1=ap$ for some $a\in \BZ_+$ with $\gcd(a,p)=1$, 
we classify the irreducible $\BN$-graded modules for the simple quotient of $V_{L,\BF}$.

\subsection{The simple quotient of $V_{L,\BF}$}

We first recall an (abstract) bialgebra $\B$ (cf. \cite{B1}).
Let $\B$ be a vector space with a designated basis $\{\D^{(k)}\mid k\in \BN\}$.
Then $\B$ is a  bialgebra with
\begin{align}
&\D^{(j)}\cdot\D^{(k)}=\binom{j+k}{j}\D^{(j+k)},\;\D^{(0)}=1,\\
&\Delta(\D^{(k)})=\sum_{r=0}^{k}\D^{(k-r)}\otimes \D^{(r)},\;\varepsilon(\D^{(k)})=\delta_{k,0}
\end{align}
for $j,k\in \BN$.
It is known (see \cite{LM}; cf. \cite{B1}) that a vertex 
algebra $V$ carries a $\B$-module structure such that 
\begin{equation}\label{eq:B-mod-act}
\D^{(k)}v=v_{-k-1}\mathbf{1}\;\text{ for } k\in \BN \text{ and } v\in V.
\end{equation}

We now turn to another bialgebra $\mathcal{H}$ (see \cite{LM2}).
Consider the $3$-dimensional simple Lie algebra $\mathfrak{sl}_2$ 
over $\BC$ with a basis $\{L_{-1}, L_{0}, L_{1}\}$ such that
\begin{equation*}
[L_1,L_{-1}]=2L_{0},\;[L_{0},L_{\pm 1}]=\mp L_{\pm 1}.
\end{equation*}
The universal enveloping algebra $U(\mathfrak{sl}_2)$ is naturally 
a Hopf algebra, hence also a bialgebra over $\BC$.
For $k\in \BN$, set 
\begin{equation}
L_{\pm 1}^{(k)}:=\frac{(L_{\pm 1})^{k}}{k!},\;L_{0}^{(k)}:=\binom{-2L_0}{k}=\frac{(-2L_0)(-2L_0-1)\cdots (-2L_0-k+1)}{k!}
\end{equation}
in $U(\mathfrak{sl}_2)$.
Let $U(\mathfrak{sl}_2)_{\BZ}$ be the $\BZ$-span in $U(\mathfrak{sl}_2)$ of elements
\begin{equation}\label{eq:basis}
L_{-1}^{(i)}L_{0}^{(j)}L_{1}^{(k)}\;\text{ for }i,j,k\in \BN.
\end{equation}
By a result of Kostant (cf. \cite[\S 26]{H}), 
$U(\mathfrak{sl}_2)_{\BZ}$ is an integral form of $U(\mathfrak{sl}_2)$ as a Hopf algebra.
For any $j,k\in\BN$, the following relations hold in $U(\mathfrak{sl}_2)_{\BZ}$:
\begin{align}
L_{0}^{(j)}L_{0}^{(k)}&=\sum_{t=0}^{j}\binom{j}{t}\binom{k+t}{j}L_{0}^{(k+t)},\label{eq:alg-str}\\
L_{\pm 1}^{(j)}L_{\pm 1}^{(k)}&=\binom{j+k}{j}L_{\pm 1}^{(j+k)},\\
L_{0}^{(j)}L_{\pm 1}^{(k)}&=\sum_{t=0}^{j}\binom{\pm 2k}{t}L_{\pm 1}^{(k)}L_{0}^{(j-t)},\\
L_{1}^{(j)}L_{-1}^{(k)}&=\sum_{s=0}^{\min(j,k)}\sum_{t=0}^{s}\binom{-j-k+2s}{t}(-1)^{s}L_{-1}^{(k-s)}L_{0}^{(s-t)}L_{1}^{(j-s)},\\
L_{\pm 1}^{(j)}L_{0}^{(k)}&=\sum_{t=0}^{k}\binom{\mp2j}{t}L_{0}^{(k-t)}L_{\pm 1}^{(j)}.\nonumber
\end{align}
Furthermore, the coalgebra structure on $U(\mathfrak{sl}_2)_{\BZ}$ is given by 
\begin{align}\label{eq:bialg-str}
\varepsilon\big( L_{r}^{(k)}\big)=\delta_{k,0},\;
\Delta\big( L_{r}^{(k)}\big)=\sum_{i=0}^{k}L_{r}^{(k-i)}\otimes L_{r}^{(i)}
\end{align}
for $r\in \{1,0,-1\}$ and $k\in \BN$.
Following \cite{LM2}, define 
\begin{equation*}
\mathcal{H}:=\BF\otimes_{\BZ}U(\mathfrak{sl}_2)_{\BZ},
\end{equation*}
which is a Hopf algebra over $\BF$.
Moreover, the elements of the form \eqref{eq:basis} form a basis of $\mathcal{H}$, 
and the bialgebra structure is given by \eqref{eq:alg-str}--\eqref{eq:bialg-str}.

Note that $\mathcal{H}$ is a $\BZ$-graded algebra with
\begin{equation}
\deg L_{\pm 1}^{(k)}=\mp k,\;\deg L_0^{(k)}=0\;\text{ for }k\in \BN.
\end{equation}
A $\BZ$-graded $\mathcal{H}$-module is defined in the usual way.
Furthermore, recall from \cite[Definition 3.2]{LM2} that
a {\em $\BZ$-graded weight $\mathcal{H}$-module} is a $\BZ$-graded 
$\mathcal{H}$-module $W=\bigoplus_{m\in\BZ} W_{(m)}$ on which $L_{0}^{(k)}$ 
acts as $\binom{-2 \deg }{k}$ for $k\in \BN$, i.e.,
\begin{equation}
L_{0}^{(k)}|_{W_{(m)}}=\binom{-2m}{k}\;\text{ for }k\in \BN, m\in \BZ.
\end{equation}

We then recall $\mathcal{H}$-module vertex algebras over $\BF$, 
which are analogues of quasi-vertex algebras over a field of 
characteristic zero (see \cite[\S 2.8]{FHL}).

\begin{definition}[{\cite[Definition~3.7]{LM2}}]\label{def:H-mod-va}
An {\em $\mathcal{H}$-module vertex algebra} is a $\BZ$-graded vertex 
algebra $V=\bigoplus_{m\in\BZ}V_{(m)}$ which is also a $\BZ$-graded 
weight $\mathcal{H}$-module (with the natural $\B$-module action) 
satisfying the following conditions:
\begin{enumerate}[(i)]
\item $V_{(m)}=0$ for $m$ sufficiently negative.
\item $L_{1}^{(k)}\mathbf{1}=\varepsilon\big( L_{1}^{(k)}\big)\mathbf{1}=\delta_{k,0}\mathbf{1}$ for $k\in\BN$.
\item  For homogeneous $v\in V$,
\begin{equation}
e^{xL_1}Y(v,x_0)e^{-xL_1}=Y\big(e^{x(1-xx_0)L_1}(1-xx_0)^{-2\deg}v,x_0/(1-xx_0)\big).
\end{equation} 
\end{enumerate}
\end{definition}

Let $L$ be the root lattice of type $A_n$ with $n+1\equiv 0\pmod{p}$.
We now show that $V_{L,\BF}$ is an $\mathcal{H}$-module vertex algebra.
Recall that in characteristic zero, $(V_{L},Y,\mathbf{1},\omega)$ is 
a vertex operator algebra with
$\omega=\frac{1}{2}\sum_{i=1}^{n}\alpha_i(-1)\beta^{i}(-1)\mathbf{1}$,
where $\{\beta^1,\ldots,\beta^n\}$ is the dual basis of 
$\{\alpha_1,\ldots,\alpha_n\}$ with respect to $\langle\cdot,\cdot\rangle$.
Using \cite[(6.4.76)]{LL}, we have
\begin{equation}\label{eq:L-1}
L(1)=\Res_{x}x^2Y(\omega,x)=\frac{1}{2}\sum_{i=1}^{n}\sum_{t\in \BZ}\alpha_i(t)\beta^i(-t+1).
\end{equation}
For $\alpha\in L$, since $h(m)\iota(e_{\alpha})=0$ for $h\in \fh$ and 
$m\in \BZ_+$, it follows from \eqref{eq:L-1} that $L(1)\iota(e_{\alpha})=0$.
Then 
\begin{equation}\label{eq:e-L1}
e^{xL(1)}\iota(e_{\alpha})=\iota(e_{\alpha}).
\end{equation}

For $\gamma_1,\ldots,\gamma_k\in \pm \Pi$, by \cite[(5.2.38)]{FHL}, 
\eqref{eq:e-L1}, and $\wt \iota(e_{\gamma_i})=\frac{\langle\gamma_i,\gamma_i\rangle}{2}$
for $1\le i\le k$, 
we obtain
\begin{align*}
e^{xL(1)}Y(\iota(e_{\gamma_i}),x_0)e^{-xL(1)}&=Y\big(e^{x(1-xx_0)L(1)}(1-xx_0)^{-2L(0)}\iota(e_{\gamma_i}),x_0/(1-xx_0)\big)\nonumber\\
&=(1-xx_0)^{-\langle\gamma_i,\gamma_i\rangle}Y(\iota(e_{\gamma_i}),x_0/(1-xx_0)).
\end{align*}
Consequently,
\begin{align}\label{eq:L-1-r-act}
&\?e^{xL(1)}Y(\iota(e_{\gamma_1}),x_1)\cdots Y(\iota(e_{\gamma_k}),x_{k})\mathbf{1}\nonumber\\
&=e^{xL(1)}Y(\iota(e_{\gamma_1}),x_1)(e^{-xL(1)} e^{xL(1)})\cdots (e^{-xL(1)}e^{xL(1)})Y(\iota(e_{\gamma_k}),x_{k})( e^{-xL(1)} e^{xL(1)})\mathbf{1}\nonumber \\
&=\big(e^{xL(1)}Y(\iota(e_{\gamma_1}),x_1)e^{-xL(1)}\big) e^{xL(1)}\cdots e^{-xL(1)}\big(e^{xL(1)}Y(\iota(e_{\gamma_k}),x_{k})e^{-xL(1)}\big)\big(e^{xL(1)}\mathbf{1}\big)\nonumber \\
&=\Big(\prod_{i=1}^{k}(1-xx_i)^{-\langle\gamma_i,\gamma_i\rangle}\Big)Y(\iota(e_{\gamma_1}),x_1/(1-xx_1))\cdots Y(\iota(e_{\gamma_k}),x_k/(1-xx_k))\mathbf{1}.
\end{align}

We now return to the case of characteristic $p$.
By Theorem \ref{th:Mu-th-1} and \cite[Proposition 3.9.3]{LL}, 
each element $v\in V_{L,\BF}$ is an $\BF$-linear combination of 
coefficients of products of the form 
$Y(\iota(e_{\gamma_1}),x_1)\cdots Y(\iota(e_{\gamma_k}),x_k)\mathbf{1}$, 
where $\gamma_1, \ldots, \gamma_k\in \pm \Pi$ and $k\in \BN$.
Since $L$ is an even lattice, the coefficients in \eqref{eq:L-1-r-act} are integers.
Thus, \eqref{eq:L-1-r-act} still holds in $V_{L,\BF}$.
This shows that 
$L(1)^{(m)}V_{L,\BZ}\subset V_{L,\BZ}$.
Therefore, $V_{L,\BF}$ carries an $\mathcal{H}$-module vertex algebra structure with
\begin{equation*}
L_{-1}^{(m)}= \D^{(m)},\quad L_{0}^{(m)}= \binom{-2\deg }{m},\quad L_{1}^{(m)} = L(1)^{(m)}\text{ for } m\in \BN.
\end{equation*}

Next, we show that $V_{L,\BF}$ has a nonzero maximal $\BN$-graded ideal.
Recall from \cite[Proposition 3.1]{L} that there exists a unique nondegenerate 
symmetric invariant bilinear form $(\cdot,\cdot)_{V_{L}}$ on $V_{L}$ 
over $\BC$ such that $(\1,\1)_{V_{L}}=1$.
It is characterized by the following conditions (see \cite[Theorem 3.3]{DG}):
\begin{enumerate}[(1)]
\item $\big(\iota(e_{\alpha}), \iota(e_{\beta})\big)_{V_{L}}=\delta_{\alpha,-\beta}$ for $\alpha, \beta\in L$.
\item $\big(\alpha(m)u,v\big)_{V_{L}}=-\big(u,\alpha(-m)v\big)_{V_{L}}$ for all $u,v \in V_{L}$, $\alpha\in L$, and $m\in \BN$.
\end{enumerate}
Then a symmetric bilinear form on $V_{L,\BF}$, 
denoted by $(\cdot,\cdot)_{V_{L,\BF}}$, is defined as follows (see \cite{Mu}):
\begin{enumerate}[(1)]
\item $\big(\iota(e_{\alpha}), \iota(e_{\beta})\big)_{V_{L,\BF}}=\delta_{\alpha,-\beta}$ for $\alpha, \beta\in L$.
\item $\big(s_{\alpha,m}u,v\big)_{V_{L,\BF}}=\big(u,r_{\alpha,m}v\big)_{V_{L,\BF}}$ for all $u,v \in V_{L,\BF}$, $\alpha\in L$, and $m\in \BN$.
\end{enumerate}
The bilinear form $(\cdot,\cdot)_{V_{L,\BF}}$ is well defined 
since $(V_{L,\BZ}, V_{L,\BZ})_{V_{L}}\subset \BZ$ by \cite[Theorem 3.3]{DG}.

Set
\begin{equation*}
\Rad (\cdot,\cdot)_{V_{L,\BF}}:=\{v\in V_{L,\BF}\mid (v,u)_{V_{L,\BF}}=0 \text{ for all } u\in V_{L,\BF}\}.
\end{equation*}
The following results are straightforward.
\begin{lemma}\label{lem:some-facts-Rad}
	\begin{enumerate}[(i)]
		\item {\normalfont\cite[Lemma 7]{Mu}} If $v\in \Rad(\cdot,\cdot)_{V_{L,\BF}}$, 
    then $r_{\alpha,m}v$ and $s_{\alpha,m}v$ lie in  $\Rad(\cdot,\cdot)_{V_{L,\BF}}$ 
    for all $\alpha\in L$ and $m\in\BN$.
		\item $\Rad(\cdot,\cdot)_{V_{L,\BF}}$ is an $L\times \BN$-graded subspace of $V_{L,\BF}$.
		\item For $u\in M(1)_{\BF}$ and $\alpha\in L$, we have $u\otimes \iota(e_{\alpha})\in \Rad(\cdot,\cdot)_{V_{L,\BF}}$ if and only if $u\in \Rad(\cdot,\cdot)_{V_{L,\BF}}$.
	\end{enumerate}
\end{lemma}

We then obtain the following result.

\begin{lemma}\label{lem:Rad}
The subspace $\Rad(\cdot,\cdot)_{V_{L,\BF}}$ is a nontrivial ideal of $V_{L,\BF}$.
\end{lemma}

\begin{proof}
It follows from Lemma \ref{lem:some-facts-Rad} that 
$\Rad (\cdot,\cdot)_{V_{L,\BF}}$ is a $V_{L,\BF}$-submodule of $V_{L,\BF}$, 
i.e., a left ideal of $V_{L,\BF}$.
Note that the bilinear form $(\cdot,\cdot)_{V_{L,\BF}}$ is 
obtained by reduction modulo $p$ from the unique invariant 
bilinear form $(\cdot,\cdot)_{V_L}$ on $V_{L}$.
Since $V_{L,\BF}$ is an $\mathcal{H}$-module vertex algebra,
we have (see \cite[Remark 4.2]{LM2})
\begin{equation*}
\big(\D^{(m)}u,v\big)_{V_{L,\BF}}=\big(L_{-1}^{(m)}u,v\big)_{V_{L,\BF}}=\big(u,L_{1}^{(m)}v\big)_{V_{L,\BF}}
\end{equation*}
for $u,v\in V_{L,\BF}$ and $m\in \BN$.
This implies that $\Rad(\cdot,\cdot)_{V_{L,\BF}}$ is $\B$-stable.
It follows from \cite[Remark 3.9.8]{LL} (see also \cite[\S 2]{JLM}) 
that $\Rad(\cdot,\cdot)_{V_{L,\BF}}$ is an ideal of $V_{L,\BF}$.

Since $\det G_{L}\equiv 0\pmod{p}$, we see that the system of 
homogeneous linear equations $G_{L}X=0$ over $\BF$ has nonzero solutions.
Note that $X=(1,2,\ldots,n)^{T}$ is such a solution.
By \cite[Theorem 5]{Mu}, we have $0\ne \sum_{i=1}^{n}is_{\alpha_i,1}\in \Rad(\cdot,\cdot)_{V_{L,\BF}}$.
Thus, $\Rad(\cdot,\cdot)_{V_{L,\BF}}\ne 0$.
On the other hand, for any $\alpha\in L$, we have $(\iota(e_{\alpha}),\iota(e_{-\alpha}))_{V_{L,\BF}}=1$, so $\iota(e_{\alpha})\notin \Rad(\cdot,\cdot)_{V_{L,\BF}}$.
Consequently, $\Rad(\cdot,\cdot)_{V_{L,\BF}}\subsetneq V_{L,\BF}$.
Therefore, $\Rad(\cdot,\cdot)_{V_{L,\BF}}$ is a nontrivial ideal of $V_{L,\BF}$.
\end{proof}

Therefore, the bilinear form $(\cdot,\cdot)_{V_{L,\BF}}$ induces 
a bilinear form $(\cdot,\cdot)_{\bar V_{L,\BF}}$ on 
$V_{L,\BF}/\Rad(\cdot,\cdot)_{V_{L,\BF}}$, defined by
\begin{equation*}
	(u+\Rad(\cdot,\cdot)_{V_{L,\BF}}, v+\Rad(\cdot,\cdot)_{V_{L,\BF}})_{\bar V_{L,\BF}} = (u,v)_{V_{L,\BF}}
\end{equation*}
for $u,v\in V_{L,\BF}$.
It is clear that $(\cdot,\cdot)_{\bar V_{L,\BF}}$ is nondegenerate 
on $V_{L,\BF}/\Rad(\cdot,\cdot)_{V_{L,\BF}}$.
By Lemmas \ref{lem:some-facts-Rad} and \ref{lem:Rad}, 
$V_{L,\BF}/\Rad(\cdot,\cdot)_{V_{L,\BF}}$ is an $L\times \BN$-graded 
vertex algebra, and hence an $L\times \BN$-graded 
$V_{L,\BF}/\Rad(\cdot,\cdot)_{V_{L,\BF}}$-module.
Thus, using the same argument as in \cite[Proposition 10]{Mu} 
(see also \cite[Proposition 3.6]{ZM}), we obtain the following result.

\begin{lemma}[\cite{Mu}]\label{lem:simp-quot-latt}
The quotient $V_{L,\BF}/\Rad(\cdot,\cdot)_{V_{L,\BF}}$ is an 
irreducible module over itself and a simple vertex algebra.
\end{lemma}

In what follows, we determine the maximal graded ideal $\Rad(\cdot,\cdot)_{V_{L,\BF}}$ of $V_{L,\BF}$.
We shall use Lucas' theorem, which states that for any $m,k\in \BN$ with
\begin{equation*}
m=m_0+m_1p+\cdots+m_tp^t,\;\;k=k_0+k_1p+\cdots+k_tp^t,
\end{equation*}
where $t\in \BN$, $0\le m_i,k_i\le p-1$ for $0\le i\le t$, we have
\begin{equation}
\binom{m}{k}\equiv \prod_{i=0}^{t}\binom{m_i}{k_i}\pmod{p}.
\end{equation}
Since $n+1\equiv 0\pmod{p}$, 
we write
\begin{equation*}
n+1=(n+1)_rp^r+(n+1)_{r+1}p^{r+1}+\cdots+(n+1)_tp^t,
\end{equation*}
where $r\in \BZ_+$ is minimal such that $(n+1)_r\ne 0$.
In particular, $n+1=ap^r$ with $\gcd(a,p)=1$, 
where $a=(n+1)_{r}+(n+1)_{r+1}p+\cdots+(n+1)_{t}p^{t-r}$.
Furthermore, set 
\begin{equation}\label{eq:eta}
\eta:=\alpha_1+2\alpha_2+\cdots+n\alpha_n\in L.
\end{equation}
A direct calculation gives
\begin{equation}\label{eq:eta-alpha_i}
    \langle\eta,\alpha_i\rangle=0 \text{ for }1\le i\le n-1 \text{ and }\langle\eta,\alpha_n\rangle=n+1.
\end{equation}

\begin{lemma}\label{lem:4.4}
Let $I_{L}$ be the vector space spanned by elements of the form
\begin{equation}\label{eq:cha-IL}
s_{\alpha^1,m_1}\ldots s_{\alpha^k,m_k}s_{\eta,m}\iota(e_{\alpha}),
\end{equation}
where $\alpha\in L$, $\alpha^i\in \Pi$, $m_i\in \BZ_+$, $k\in \BN$, and $m\in \BZ_+$ with $p^r\nmid m$.
Then $I_{L}\subset \Rad(\cdot,\cdot)_{V_{L,\BF}}$.
\end{lemma}

\begin{proof}
Let $m\in \BZ_+$ with $p^r\nmid m$.
We first show that $s_{\eta,m}\in \Rad(\cdot,\cdot)_{V_{L,\BF}}$.
Let $s_{\alpha^1,m_1}\cdots s_{\alpha^k,m_k}\iota(e_{\alpha})$ be an element of $V_{L,\BF}$, 
where $\alpha^i\in \Pi$, $m_i\in \BZ_+$, $k\in \BN$, and $\alpha\in L$.
If $\alpha\ne 0$, or if $\alpha^i\ne \alpha_n$ for all $1\le i\le k$, 
or if $m_1+\cdots+m_k\ne m$, then by \eqref{eq:rs-sr} and \eqref{eq:eta-alpha_i} we have
\begin{equation*}
(s_{\eta,m},s_{\alpha^1,m_1}\cdots s_{\alpha^k,m_k}\iota(e_{\alpha}))_{V_{L,\BF}}=(\mathbf{1},r_{\eta,m}s_{\alpha^1,m_1}\cdots s_{\alpha^k,m_k}\iota(e_{\alpha}))_{V_{L,\BF}}=0.
\end{equation*}
Thus, it remains to consider the case where $\alpha=0$, 
$m_1+\cdots+m_k=m$, and $\alpha^i=\alpha_n$ for at least one $i$. 
Without loss of generality, assume that the first $l$ of 
the $\alpha^i$ equal $\alpha_n$, with $1\le l\le k$.
By \eqref{eq:rs-sr}, a direct calculation shows that
\begin{align}\label{eq:3.25}
&\?(s_{\eta,m}, s_{\alpha^1,m_1}\cdots s_{\alpha^k,m_k})_{V_{L,\BF}}\nonumber\\
&=(\mathbf{1}, r_{\eta,m}s_{\alpha^1,m_1}\cdots s_{\alpha^k,m_k})_{V_{L,\BF}}\nonumber\\
&=(\mathbf{1},\sum_{j_1,\ldots,j_k\in \BN}(-1)^{j_1+\cdots+j_k}\bigg(\prod_{i=1}^{k}\binom{\langle\eta,\alpha^i\rangle}{j_i}\bigg)s_{\alpha^1,m_1-j_1}\cdots s_{\alpha^k,m_k-j_k}r_{\eta,m-(j_1+\cdots+j_k)})_{V_{L,\BF}}\nonumber\\
&=(\mathbf{1},\sum_{\substack{j_1,\ldots,j_k\in \BN,\\
j_1+\cdots+j_k=m}}(-1)^{m}\binom{n+1}{j_1}\cdots\binom{n+1}{j_{l}}\binom{0}{j_{l+1}}\cdots \binom{0}{j_k}s_{\alpha^1,m_1-j_1} \cdots s_{\alpha^k,m_k-j_{k}})_{V_{L,\BF}}.
\end{align}
If $l<k$, then $j_{l+1}=\cdots=j_k=0$ and $j_1+\cdots+j_k=m$ 
force $j_1+\cdots+j_l=m$. 
Combined with $m_1+\cdots+m_l<m$, this implies that there 
exists an $i$ with $1\le i\le l$ such that $m_i-j_i<0$.
Since $s_{\alpha,t}=0$ for $t<0$, 
we obtain $(s_{\eta,m}, s_{\alpha^1,m_1}\cdots s_{\alpha^k,m_k})_{V_{L,\BF}}=0$.
If $l=k$, then by \eqref{eq:3.25} and $n+1=ap^r$, we have
\begin{equation*}
(s_{\eta,m}, s_{\alpha^1,m_1}\cdots s_{\alpha^k,m_k})_{V_{L,\BF}}=
(\mathbf{1},\sum_{\substack{j_1,\ldots,j_k\in \BN,\\
j_1+\cdots+j_k=m}}(-1)^{m}\bigg(\prod_{i=1}^{k}\binom{ap^r}{j_i}\bigg)s_{\alpha^1,m_1-j_1}\cdots s_{\alpha^k,m_k-j_k})_{V_{L,\BF}}.
\end{equation*}
Since $p^r\nmid m$, we see that there exists an $i$ with $1\le i\le k$ such that $p^r\nmid j_i$.
Then by Lucas' theorem, $\prod_{i=1}^{k}\binom{ap^r}{j_i}\equiv 0\pmod{p}$, 
and hence $(s_{\eta,m}, s_{\alpha^1,m_1}\cdots s_{\alpha^k,m_k})_{V_{L,\BF}}=0$.
Therefore, $s_{\eta,m}\in \Rad(\cdot,\cdot)_{V_{L,\BF}}$.

By Lemma \ref{lem:some-facts-Rad}(i), 
we have $s_{\alpha^1,m_1}\ldots s_{\alpha^k,m_k}s_{\eta,m}\in \Rad(\cdot,\cdot)_{V_{L,\BF}}$
for $\alpha^i\in \Pi$, $m_i\in \BZ_+$, 
$k\in \BN$, and $m\in \BZ_+$ with $p^r\nmid m$.
Then by Lemma \ref{lem:some-facts-Rad}(iii), we have
\begin{equation*}
s_{\alpha^1,m_1}\ldots s_{\alpha^k,m_k}s_{\eta,m}\iota(e_{\alpha})\in \Rad(\cdot,\cdot)_{V_{L,\BF}}
\end{equation*}
for  $\alpha\in L$.
Thus, $I_{L}\subset  \Rad(\cdot,\cdot)_{V_{L,\BF}}$.
\end{proof}

Furthermore, we have the following result.
\begin{lemma}\label{lem:I-L-ideal}
The subspace $I_{L}$ is the $V_{L,\BF}$-submodule generated 
by $\{s_{\eta,m}\mid m\in \BZ_+ \text{ with } p^r\nmid m\}$.
Moreover, $I_{L}$ is an ideal of $V_{L,\BF}$.
\end{lemma}

\begin{proof}
Let $U_{L}$ be the $V_{L,\BF}$-module generated by 
$\{s_{\eta,m}\mid m\in \BZ_+ \text{ with } p^r\nmid m\}$.
By \cite[Propositions 4.5.6, 3.9.3]{LL} and Theorem \ref{th:Mu-th-1}, 
$U_{L}$ is spanned by the coefficients of products of the form
\begin{equation*}
Y(\iota(e_{\gamma_1}),x_1)\cdots Y(\iota(e_{\gamma_k}),x_k)s_{\eta,m},
\end{equation*}
where $\gamma_1,\ldots,\gamma_k\in \pm \Pi$, $k\in \BN$, 
and $m\in \BZ_+$ with $p^r\nmid m$.
By \eqref{eq:def-E+E-}, \eqref{eq:rs-sr}, 
and $r_{\gamma_i,q}\mathbf{1}=0$ for $q\in \BZ_+$, we have
\begin{align}\label{eq:E+-s-eta}
E^{+}(-\gamma_i,x)s_{\eta,m}&=\sum_{t\in \BN}r_{\gamma_i,t}s_{\eta,m}x^{-t}=\sum_{t\in \BN}\sum_{j\in \BN}(-1)^{j}\binom{\langle \gamma_i,\eta\rangle}{j}s_{\eta,m-j}r_{\gamma_i,t-j}\mathbf{1}x^{-t}\nonumber\\
&=\sum_{t\in \BN}(-1)^{t}\binom{\langle \gamma_i,\eta\rangle}{t}s_{\eta,m-t}x^{-t}.
\end{align}
Now using \eqref{eq:ver-ope}, \eqref{eq:E+E-E-E+}, \eqref{eq:E+-s-eta}, 
and $x^{\gamma_i}e_{\gamma_j}=x^{\langle\gamma_i,\gamma_j\rangle}e_{\gamma_j}x^{\gamma_i}$ 
(see \cite[Proposition 6.4.5]{LL}), we obtain
\begin{align*}
&\?Y(\iota(e_{\gamma_1}),x_1)\cdots Y(\iota(e_{\gamma_k}),x_k)s_{\eta,m}\nonumber\\
&=(E^{-}(-\gamma_1,x_1)E^{+}(-\gamma_1,x_1)e_{\gamma_1}x_1^{\gamma_1})\cdots (E^{-}(-\gamma_k,x_k)E^{+}(-\gamma_k,x_k)e_{\gamma_k}x_k^{\gamma_k})s_{\eta,m}\\
&=\Bigl(\prod_{i=1}^{k}\epsilon(\gamma_i,\sum_{i<j\le k}\gamma_j)\Bigr)\Bigl(\prod_{1\le i<j\le k}(x_i-x_j)^{\langle \gamma_i,\gamma_j \rangle}\Bigr) E^{-}(-\gamma_1,x_1)\cdots E^{-}(-\gamma_k,x_k) \\
&\?\times E^{+}(-\gamma_1,x_1)\cdots E^{+}(-\gamma_k,x_k)s_{\eta,m}\iota(e_{\gamma_1+\cdots+\gamma_k}) \\
&=\Bigl(\prod_{i=1}^{k}\epsilon(\gamma_i,\sum_{i<j\le k}\gamma_j)\Bigr)\Bigl(\prod_{1\le i<j\le k}(x_i-x_j)^{\langle \gamma_i,\gamma_j \rangle}\Bigr) E^{-}(-\gamma_1,x_1)\cdots E^{-}(-\gamma_k,x_k) \\
&\?\times\Bigl(\sum_{t_1,\ldots,t_k\in \BN}(-1)^{t_1+\cdots+t_k}\binom{\langle\gamma_1,\eta\rangle}{t_1}\cdots\binom{\langle\gamma_k,\eta\rangle}{t_k} s_{\eta,m-t_1-\cdots-t_k}\iota(e_{\gamma_1+\cdots+\gamma_k})x_1^{-t_1}\cdots x_{k}^{-t_{k}}\Bigr).
\end{align*}
Suppose that $\pm\alpha_n$ occurs $l$ times among $\gamma_1,\ldots,\gamma_k$.
Without loss of generality, assume that the first $l$ are $\pm \alpha_n$.
By \eqref{eq:eta-alpha_i}, we have
\begin{align}\label{eq:0-l-k}
&\? Y(\iota(e_{\gamma_1}),x_1)\cdots Y(\iota(e_{\gamma_k}),x_k)s_{\eta,m}\nonumber\\
&=\Big(\prod_{i=1}^{k}\epsilon(\gamma_i,\sum_{i<j\le k}\gamma_j)\Big)\Big(\prod_{1\le i<j\le k}(x_i-x_j)^{\langle \gamma_i,\gamma_j \rangle}\Big) E^{-}(-\gamma_1,x_1)\cdots E^{-}(-\gamma_k,x_k)\nonumber\\
&\?\times\bigg(\sum_{t_1,\ldots,t_l\in \BN}(-1)^{t_1+\cdots+t_l}\binom{\pm ap^r}{t_1}\cdots\binom{\pm ap^r}{t_l} s_{\eta,m-t_1-\cdots-t_l}\iota(e_{\gamma_1+\cdots+\gamma_k})x_1^{-t_1}\cdots x_{l}^{-t_{l}}\bigg)\nonumber\\
&=\Big(\prod_{i=1}^{k}\epsilon(\gamma_i,\sum_{i<j\le k}\gamma_j)\Big)\Big(\prod_{1\le i<j\le k}(x_i-x_j)^{\langle \gamma_i,\gamma_j \rangle}\Big) E^{-}(-\gamma_1,x_1)\cdots E^{-}(-\gamma_k,x_k)\nonumber\\
&\?\times\bigg(\sum_{\substack{t_1,\ldots,t_l\in \BN,\\ p^r\mid t_i\text{ for }1\le i\le l}}(-1)^{t_1+\cdots+t_l}\bigg(\prod_{i=1}^{l}\binom{\pm ap^r}{t_i}\bigg) s_{\eta,m-t_1-\cdots-t_l}\iota(e_{\gamma_1+\cdots+\gamma_k})x_1^{-t_1}\cdots x_{l}^{-t_{l}}\bigg),
\end{align}
as $(-1)^{t}\binom{\pm ap^r}{t}\equiv 0\pmod{p}$ for $p^r\nmid t$.
Here, $p^r\nmid (m-t_1-\cdots-t_l)$ since $p^r\nmid m$ and 
$p^r\mid t_i$ for all $i$.
Recall from \cite[Proposition 6.4.12]{LL} and \cite[(5.2.14)]{FLM} 
that $\epsilon(\alpha,\gamma)=\pm 1$ for $\alpha,\gamma\in L$.
Then combining Lemma \ref{lem:DG-s--alpha} with \eqref{eq:0-l-k}, 
we obtain $U_{L}\subset I_{L}$.

Conversely, we prove $I_{L}\subset U_{L}$.
Let $\gamma\in L$, and let $\gamma_1,\ldots,\gamma_k\in \pm \Pi$ 
such that $\gamma=\gamma_1+\cdots+\gamma_k$.
We first use induction on the nonnegative integer $b$ to show that 
\begin{equation}\label{eq:desire}
s_{\gamma_1,m_1}\cdots s_{\gamma_{k},m_k} s_{\eta,m}\iota(e_{\gamma})\in U_{L}
\end{equation}
for all $m_1,\ldots,m_k\in \BN$, $m\in\BZ_+$ with $bp^r<m <(b+1)p^r$.

By \eqref{eq:0-l-k}, we have
\begin{align}\label{eq:3.29}
&E^{-}(-\gamma_1,x_1)\cdots E^{-}(-\gamma_k,x_k)\biggl(\sum_{\substack{t_1,\ldots,t_l\in \BN,\\ p^r\mid t_i\text{ for }1\le i\le l}}(-1)^{t_1+\cdots+t_l}\biggl(\prod_{i=1}^{l}\binom{\pm ap^r}{t_i}\biggr) s_{\eta,m-t_1-\cdots-t_l}\nonumber\\
&\?\times\iota(e_{\gamma})x_1^{-t_1}\cdots x_{l}^{-t_{l}}\biggr)\nonumber\\
&=\Big(\prod_{i=1}^{k}\epsilon(\gamma_i,\sum_{i<j\le k}\gamma_j)^{-1}\Big)\Big(\prod_{1\le i<j\le k}(x_i-x_j)^{-\langle \gamma_i,\gamma_j \rangle}\Big)Y(\iota(e_{\gamma_1}),x_1)\cdots Y(\iota(e_{\gamma_k}),x_k)s_{\eta,m}.
\end{align}
Then, taking the coefficient of $x_1^{m_1}\cdots x_k^{m_k}$ for 
$m_1,\ldots,m_k\in \BN$ in \eqref{eq:3.29}, we obtain
\begin{align}\label{eq:sum}
&\sum_{\substack{t_1,\ldots,t_l\in \BN,\\ p^r\mid t_i\text{ for }1\le i\le l}}(-1)^{t_1+\cdots+t_l}\bigg(\prod_{i=1}^{l}\binom{\pm ap^r}{t_i}\bigg)s_{\gamma_1,m_1+t_1}\cdots s_{\gamma_l,m_l+t_l}s_{\gamma_{l+1},m_{l+1}}\cdots s_{\gamma_{k},m_{k}}s_{\eta,m-t_1-\cdots-t_l}\iota(e_{\gamma})\nonumber\\
&=\Res_{x_1}x_1^{-m_1-1}\cdots \Res_{x_k}x_k^{-m_k-1}\Big(\Big(\prod_{i=1}^{k}\epsilon(\gamma_i,\sum_{i<j\le k}\gamma_j)^{-1}\Big)\Big(\prod_{1\le i<j\le k}(x_i-x_j)^{-\langle \gamma_i,\gamma_j \rangle}\Big)\nonumber \\
&\?\times Y(\iota(e_{\gamma_1}),x_1)\cdots Y(\iota(e_{\gamma_k}),x_k)s_{\eta,m}\Big)\in U_{L}.
\end{align}
Consider the case $0<m< p^r$ (i.e., $b=0$).
If $t_i>0$ for some $1\le i\le l$, then $t_i\ge p^r$ as 
$p^r\mid t_i$, and thus $m-t_1-\cdots-t_l<0$.
Since $s_{\eta,q}=0$ for $q<0$, we obtain
\begin{align*}
&\sum_{\substack{t_1,\ldots,t_l\in \BN,\\ p^r\mid t_i\text{ for }1\le i\le l,\\ (t_1,\ldots,t_l)\ne (0,\ldots,0)}}(-1)^{t_1+\cdots+t_l}\biggl(\prod_{i=1}^{l}\binom{\pm ap^r}{t_i}\biggr)s_{\gamma_1,m_1+t_1}\cdots s_{\gamma_l,m_l+t_l}s_{\gamma_{l+1},m_{l+1}}\cdots s_{\gamma_{k},m_{k}} \nonumber\\
& \?\times s_{\eta,m-t_1-\cdots-t_l}\iota(e_{\gamma})=0.
\end{align*}
Consequently, \eqref{eq:sum} implies that \eqref{eq:desire} holds for $0<m<p^r$.
Assume that \eqref{eq:desire} holds whenever $cp^r<m<(c+1)p^r$
for some integer $c$ with $0\le c<b$ (equivalently, for all $m\in \BZ_+$ with $p^r\nmid m$ and $m<bp^r$).
Consider the case $bp^r< m<(b+1)p^r$.
By \eqref{eq:sum}, we have
\begin{align*}
&\?s_{\gamma_1,m_1}\cdots s_{\gamma_l,m_l}s_{\gamma_{l+1},m_{l+1}}\cdots s_{\gamma_{k},m_{k}} s_{\eta,m}\iota(e_{\gamma})\nonumber\\
&=\Res_{x_1}x_1^{-m_1-1}\cdots \Res_{x_k}x_k^{-m_k-1}\nonumber \\
&\?\times\Big(\Big(\prod_{i=1}^{k}\epsilon(\gamma_i,\sum_{i<j\le k}\gamma_j)^{-1}\Big)\Big(\prod_{1\le i<j\le k}(x_i-x_j)^{-\langle \gamma_i,\gamma_j \rangle}\Big)Y(\iota(e_{\gamma_1}),x_1)\cdots Y(\iota(e_{\gamma_k}),x_k)s_{\eta,m}\Big)\nonumber\\
&\?-\Bigl(\sum_{\substack{t_1,\ldots,t_l\in \BN,\\ p^r\mid t_i\text{ for }1\le i\le l,\\ (t_1,\ldots,t_l)\ne (0,\ldots,0)}}(-1)^{t_1+\cdots+t_l}\biggl(\prod_{i=1}^{l}\binom{\pm ap^r}{t_i}\biggr)s_{\gamma_1,m_1+t_1}\cdots s_{\gamma_l,m_l+t_l}s_{\gamma_{l+1},m_{l+1}}\cdots s_{\gamma_{k},m_{k}}\nonumber\\
&\?\?\times  s_{\eta,m-t_1-\cdots-t_l}\iota(e_{\gamma})\Bigr).
\end{align*}
Since
$m-t_1-\cdots-t_l<bp^r$ and $p^r\nmid (m-t_1-\cdots-t_l)$, 
it follows from the induction hypothesis that \eqref{eq:desire} 
holds for $bp^r<m<(b+1)p^r$.
This completes the induction step.

Let $\gamma\in L$, and let $\alpha^1,\ldots,\alpha^k\in \Pi$.
Let $\gamma^1,\ldots,\gamma^t\in \pm\Pi$ such that 
$\gamma=(\alpha^1+\cdots+\alpha^k)+(\gamma^1+\cdots+\gamma^t)$.
Then \eqref{eq:desire} shows that
\begin{equation*}
	s_{\alpha^1,m_1}\cdots s_{\alpha^k,m_k}s_{\gamma^1,m_1'}\cdots s_{\gamma^t,m_t'}s_{\eta,m}\iota(e_{\gamma})\in U_{L}
\end{equation*}
for $m_i,m_j' \in \BN$, and $m \in \BZ_+$ with $p^r \nmid m$.
Taking $m_1'=\cdots=m_t'= 0$ and using $s_{\gamma^j,0}=1$,
we obtain
\begin{equation*}
	s_{\alpha^1,m_1}\cdots s_{\alpha^k,m_k} s_{\eta,m}\iota(e_{\gamma})\in U_{L}
\end{equation*}
for $\gamma \in L$, $\alpha^i\in\Pi$, $m_i\in\BZ_+$, 
$k\in\BN$, and $m\in\BZ_+$ with $p^r\nmid m$.
Therefore, $I_{L} = U_{L}$.

Now we prove that $I_{L}$ is in fact an ideal of $V_{L,\BF}$.
Since $I_{L}$ is a $V_{L,\BF}$-module,
\cite[Remark 3.9.8]{LL} (see also \cite[Section 2]{JLM})
implies that
it suffices to show that $I_{L}$ is $\B$-stable.
For any $k\in \BN$ and any $m\in \BZ_+$ with $p^r\nmid m$, 
by \eqref{eq:B-mod-act} and \eqref{eq:Y-s-alpha-n}, we have
\begin{equation}\label{eq:D-k-s-delta-m}
\D^{(k)}s_{\eta,m}=(s_{\eta,m})_{-k-1}\mathbf{1}=
\sum_{l=m}^{k+m}\binom{l}{m}s_{\eta,l}s_{-\eta,k-l+m}.
\end{equation}
First consider the terms with $p^r\nmid l$ in \eqref{eq:D-k-s-delta-m}.
Lemma \ref{DG3.3}(i) implies that $s_{-\eta,k-l+m}$ is a linear 
combination of elements in $\mathcal{S}$.
Then by \eqref{eq:cha-IL}, we have
$\binom{l}{m}s_{\eta,l}s_{-\eta,k-l+m}\in I_{L}$.
We now turn to the remaining terms with $p^r\mid l$ in \eqref{eq:D-k-s-delta-m}.
Since $p^r\nmid m$, by Lucas' theorem, 
we have $\binom{l}{m}\equiv 0\pmod{p}$, 
and thus $\binom{l}{m}s_{\eta,l}s_{-\eta,k-l+m}=0\in I_{L}$.
It follows from \eqref{eq:D-k-s-delta-m} 
that $\D^{(k)}s_{\eta,m}\in I_{L}$ for $k\in \BN$ and $m\in \BZ_+$ with $p^r\nmid m$.
Consequently, the desired result follows from the fact 
that $I_{L}$ is generated by $\{s_{\eta,m}\mid m\in \BZ_+\text{ with } p^r\nmid m\}$ 
as a $V_{L,\BF}$-module.
\end{proof}

Let $\mathcal{P}$ denote the set of all partitions, 
including the empty partition $\varnothing$.
For each $\beta\in L^*$, $\lambda=(\lambda_1,\ldots,\lambda_k)\in \mathcal{P}$ 
with $\lambda_1\ge \cdots\ge \lambda_{k}>0$, and $t\in \BN$, we define the 
vectors $s_{\lambda,p^t}(\beta),r_{\lambda,p^t}(\beta)$ 
by the following $k\times k$ determinants:
\begin{equation}\label{eq:s-Lambda-pt}
s_{\lambda,p^t}(\beta):=\det (s_{\beta,p^{t}(\lambda_i+j-i)}), \;\;r_{\lambda,p^t}(\beta):=\det (r_{\beta,p^{t}(\lambda_i+j-i)})
\end{equation}
where $s_{\beta,m}=r_{\beta,m}=0$ if $m<0$.
Set $s_{\varnothing,p^t}(\beta)=r_{\varnothing,p^t}(\beta)=1$.
In particular, we write $s_{\lambda}(\beta)$ for $s_{\lambda,1}(\beta)$, as in \cite{DG}.
For $\alpha\in L$, it is easy to see that 
\begin{equation}\label{eq:r-s-close}
(s_{\lambda,p^t}(\alpha)u,v)_{V_{L,\BF}}=(u,r_{\lambda,p^t}(\alpha)v)_{V_{L,\BF}} 
\end{equation}
for $u,v\in V_{L,\BF}$, $\lambda\in \mathcal{P}$, and $t\in \BN$.

We now obtain the following result.

\begin{proposition}\label{prop:s-chur-basis}
\begin{enumerate}[(i)]
\item The set
 \begin{equation}\label{eq:other-V-LF-bas}
 \{s_{\alpha^1,m_1}\cdots s_{\alpha^k,m_k}s_{\eta,t_1}\cdots s_{\eta,t_l}\iota(e_{\gamma})\mid  \gamma\in L, \alpha^i\in \Pi\setminus\{\alpha_n\}, m_i,t_j\in \BZ_+, k,l\in \BN\}
 \end{equation}
 is an $\BF$-basis of $V_{L,\BF}$.
\item The set
\begin{align}\label{eq:V-L-quo-bas-1}
\{s_{\alpha^1,m_1}\cdots s_{\alpha^k,m_k}s_{\eta,m_1'}\cdots s_{\eta,m_{l}'}\iota(e_{\gamma})+I_{L}&\mid   \gamma\in L,  \alpha^i\in \Pi\setminus\{\alpha_n\}, m_i \in \BZ_+,\nonumber \\
& \?m_j'\in \BZ_+ \text{ with } p^r\mid m_j',k,l\in \BN\}
\end{align}
is an $\BF$-basis of $V_{L,\BF}/I_L$.
\item 
The set 
\begin{align}\label{eq:V-L-quo-bas-2}
\{s_{\alpha^1,m_1}\cdots s_{\alpha^k,m_k}\iota(e_{\gamma})+I_{L}\mid \gamma\in L,
\alpha^i\in \Pi, m_i\in \BZ_+ \text{ with } p^r\mid m_i \text{ if } \alpha^i=\alpha_n, k\in \BN\}
\end{align}
is an $\BF$-basis of $V_{L,\BF}/I_L$.
\item 
The set 
\begin{equation}\label{eq:V-L-quo-bas-3}
\{s_{\lambda^1}(\alpha_1)\cdots s_{\lambda^{n-1}}(\alpha_{n-1})s_{\lambda^{n},p^r}(\alpha_n)\iota(e_{\gamma})+I_{L}\mid \lambda^1,\ldots, \lambda^n\in \mathcal{P},\gamma\in L \}
\end{equation}
is an $\BF$-basis of $V_{L,\BF}/I_{L}$.
\end{enumerate}
\end{proposition}

\begin{proof}
(i) We first prove that the set
\begin{equation}\label{eq:other-M(1)F-bas}
\{s_{\alpha^1,m_1}\cdots s_{\alpha^k,m_k}s_{\eta,t_1}\cdots s_{\eta,t_l}\mid  \alpha^i\in \Pi\setminus\{\alpha_n\}, m_i,t_j\in \BZ_+, k,l\in \BN\}
\end{equation}
is an $\BF$-basis of $M(1)_{\BF}$. 
Recall that $M(1)_\BF=\bigoplus_{m\in\BN}(M(1)_\BF)_{(m)}$ 
is an $\BN$-graded vertex algebra,
where $(M(1)_\BF)_{(m)}=\operatorname{span}\{s_{\alpha^1,m_1}\cdots s_{\alpha^k,m_k}\mid \alpha^i\in \Pi, m_i\in \BZ_+, m_1+\cdots +m_k=m, k\in \BN\}$.

We define an order $\prec$ on the set $\mathcal{S}$ (see \eqref{eq:M}).
For $\lambda=(\lambda_1,\ldots,\lambda_l)\in \mathcal{P}$ 
and $\mu=(\mu_1,\ldots,\mu_t)\in \mathcal{P}$, 
we say that $\lambda$ is smaller than $\mu$ if 
$\lambda_j<\mu_j$ for some $j$ and $\lambda_i=\mu_i$ for $1\le i<j$
or if $\lambda_i=\mu_i$ for $1\le i\le l$ with $l<t$.
For $\alpha\in L$, 
we set $f_{\alpha}(\lambda):= s_{\alpha,\lambda_1}\cdots s_{\alpha,\lambda_l}$, 
and we say that $f_{\alpha}(\lambda)$ is smaller than $f_{\alpha}(\mu)$ 
if $\lambda$ is smaller than $\mu$ (see \cite[Lemma 2]{Mu}).
Let $\lambda^1,\ldots,\lambda^n\in \mathcal{P}$ and $\mu^1,\ldots,\mu^n\in \mathcal{P}$.
We say that 
\begin{equation*}
f_{\alpha_n}(\lambda^n)\cdots f_{\alpha_1}(\lambda^1) \prec f_{\alpha_n}(\mu^n)\cdots f_{\alpha_1}(\mu^1) 
\end{equation*}
if $f_{\alpha_j}(\lambda^j)$ is smaller than $f_{\alpha_j}(\mu^j)$ for some $j$, 
and $f_{\alpha_i}(\lambda^i)=f_{\alpha_i}(\mu^i)$ for all $n\ge i>j$.
Note that each element of $\mathcal{S}$ can be written as
$f_{\alpha_n}(\lambda^n)\cdots f_{\alpha_1}(\lambda^1)$ for 
some $\lambda^1,\ldots,\lambda^n \in\mathcal{P}$.
Then $\prec$ defines a total order on the set $\mathcal{S}$.
The order is compatible with multiplication: 
\begin{equation}\label{eq:ord-preserve}
    \text{if }f_1\prec f_2,\text{ then }f_1h\prec f_2h \text{ for every monomial }h.
\end{equation}

For $m\in \BZ_+$, \eqref{eq:def-E+E-} and \eqref{eq:eta} 
give the following operator identity:
\begin{align}\label{eq:s-f}
s_{\eta,m}&=\Res_{x}x^{-m-1}E^{-}(-\eta,x)\nonumber\\
&=\Res_{x}x^{-m-1}\big(E^{-}(-\alpha_1,x)E^{-}(-\alpha_{2},x)^2\cdots E^{-}(-\alpha_n,x)^{n}\big)\nonumber\\
&=ns_{\alpha_n,m}+v\nonumber\\
&=-s_{\alpha_n,m}+v,
\end{align}
where $v$ is a linear combination of degree $m$ monomials 
smaller than $s_{\alpha_n,m}$ with respect to $\prec$.

Let $\lambda^{1},\ldots,\lambda^{n}\in \mathcal{P}$ and 
$\lambda^n=(\lambda_1,\ldots,\lambda_t)$ for some $t\in \BN$.
By \eqref{eq:ord-preserve} and \eqref{eq:s-f}, we see that 
\begin{equation}\label{eq:eat=alpha+v}
f_{\eta}(\lambda^n)f_{\alpha_{n-1}}(\lambda^{n-1})\cdots f_{\alpha_1}(\lambda^1)
=(-1)^{t}f_{\alpha_n}(\lambda^n)f_{\alpha_{n-1}}(\lambda^{n-1})\cdots f_{\alpha_1}(\lambda^1)+v,
\end{equation} 
where $v$ is a linear combination of smaller monomials with respect 
to $\prec$ and each monomial is of the same degree as 
$f_{\alpha_n}(\lambda^n)f_{\alpha_{n-1}}(\lambda^{n-1})\cdots f_{\alpha_1}(\lambda^1)$.
Since $(M(1)_\BF)_{(m)}$ is finite-dimensional for every $m\in\BN$, 
and since $\mathcal{S}$ is an ordered basis of $M(1)_{\BF}$,
it follows from \eqref{eq:eat=alpha+v} that \eqref{eq:other-M(1)F-bas} 
is an $\BF$-basis of $M(1)_{\BF}$.
Therefore, \eqref{eq:other-V-LF-bas} is also an $\BF$-basis of $V_{L,\BF}$.

(ii) By (i) and by the definition of $I_{L}$ (see \eqref{eq:cha-IL}), 
\eqref{eq:V-L-quo-bas-1} forms an $\BF$-basis of $V_{L,\BF}/I_{L}$.

(iii) We define an order $\dot{\prec}$ on the set in \eqref{eq:other-M(1)F-bas}
with $\alpha_n$ replaced by $\eta$ in the definition of $\prec$.
Let $m\in\BN$. Set
\begin{align*}
A_m&=\{s_{\alpha^1,m_1}\cdots s_{\alpha^k,m_k} \mid \alpha^i\in \Pi, m_i\in \BZ_+ \text{ with } p^r\mid m_i \text{ if } \alpha^i=\alpha_n, k\in \BN, m_1+\cdots+m_k=m\},\\
B_m&=\{s_{\alpha^1,m_1}\cdots s_{\alpha^k,m_k}s_{\eta,m'_1}\cdots s_{\eta,m'_{l}} \mid \alpha^i\in \Pi\setminus\{\alpha_n\}, m_i,m'_j\in \BZ_+, p^r\mid m'_j,\ k,l\in \BN, \\
&\hspace*{22.73em}m_1+\cdots+m_k+m'_1+\cdots+m'_l=m\}.
\end{align*}
Denote by $g_1\prec g_2\prec\ldots\prec g_N$ the elements of $\mathcal{S}$
of degree $m$, listed in increasing order.
We use induction on $j$ to show that 
if $g_j=f_{\alpha_n}(\lambda^n)\cdots f_{\alpha_1}(\lambda^1)\not\in A_m$, then 
\begin{equation*}
f_{\alpha_n}(\lambda^n)f_{\alpha_{n-1}}(\lambda^{n-1})\cdots f_{\alpha_1}(\lambda^1)
\equiv v \pmod{I_L},
\end{equation*}
where $v$ is a linear combination of elements in $B_m$ smaller than 
$f_{\eta}(\lambda^n)f_{\alpha_{n-1}}(\lambda^{n-1})\cdots f_{\alpha_1}(\lambda^1)$ 
with respect to $\dot\prec$; 
if $g_j=f_{\alpha_n}(\lambda^n)f_{\alpha_{n-1}}(\lambda^{n-1})\cdots f_{\alpha_1}(\lambda^1)\in A_m$, 
then 
\begin{equation*}
f_{\alpha_n}(\lambda^n)f_{\alpha_{n-1}}(\lambda^{n-1})\cdots f_{\alpha_1}(\lambda^1)
\equiv (-1)^t f_{\eta}(\lambda^n)f_{\alpha_{n-1}}(\lambda^{n-1})\cdots f_{\alpha_1}(\lambda^1)+v \pmod{I_L},
\end{equation*}
where $v$ is a linear combination of smaller elements in $B_m$ 
with respect to $\dot\prec$.
Clearly, the induction statement holds when $\lambda^n=\varnothing$.
Consider the induction step with $\lambda^n=(\lambda_1,\ldots,\lambda_t)\ne\varnothing$. 
If $g_j=f_{\alpha_n}(\lambda^n)f_{\alpha_{n-1}}(\lambda^{n-1})\cdots f_{\alpha_1}(\lambda^1)\in A_m$,
by \eqref{eq:eat=alpha+v}, we have
\begin{equation*}
f_{\alpha_n}(\lambda^n)f_{\alpha_{n-1}}(\lambda^{n-1})\cdots f_{\alpha_1}(\lambda^1)
=(-1)^{t}f_{\eta}(\lambda^n)f_{\alpha_{n-1}}(\lambda^{n-1})\cdots f_{\alpha_1}(\lambda^1)+v,
\end{equation*}
where $v$ is a linear combination of smaller monomials with 
respect to $\prec$,
and our assertion follows from the induction hypothesis.
Now consider the case $g_j=f_{\alpha_n}(\lambda^n)\cdots f_{\alpha_1}(\lambda^1)\not\in A_m$.
It follows from \eqref{eq:cha-IL} that 
$f_{\eta}(\lambda^n)f_{\alpha_{n-1}}(\lambda^{n-1})\cdots f_{\alpha_1}(\lambda^1)\in I_L$.
Then by \eqref{eq:eat=alpha+v}, we have
\begin{equation*}
f_{\alpha_n}(\lambda^n)f_{\alpha_{n-1}}(\lambda^{n-1})\cdots f_{\alpha_1}(\lambda^1)
\equiv v \pmod{I_L},
\end{equation*}
where $v$ is a linear combination of smaller monomials with respect to $\prec$,
and our assertion follows from the induction hypothesis.
Consequently, \eqref{eq:V-L-quo-bas-2} is equivalent to \eqref{eq:V-L-quo-bas-1} 
in $V_{L,\BF}/I_L$. Since \eqref{eq:V-L-quo-bas-1} is a basis of $V_{L,\BF}/I_L$ by (ii),
it follows that \eqref{eq:V-L-quo-bas-2} is a basis of $V_{L,\BF}/I_L$.

(iv) By the argument in \cite[Proposition 3.6]{DG} 
(see also \cite[Lemma 2]{Mu}), the assertion (iv) follows from (iii).
\end{proof}

Let $\xi:=\frac{\eta}{a}\in L^*$.
Note that $\gcd(a,p)=1$.
Then by \cite[Lemma 4]{Mu},
$s_{-\xi,kp^r}$ is a well-defined element of $V_{L,\BF}$ for $k\in \BN$, 
and thus $s_{\lambda,p^r}(-\xi)$ is well defined for 
$\lambda\in \mathcal{P}$ (see \eqref{eq:s-Lambda-pt}).
We then obtain the following result (cf. \cite[Proposition 3.6]{DG}).
\begin{lemma}\label{lem:3.11}
For any $\lambda,\lambda'\in \mathcal{P}$, we have
\begin{equation*}
(s_{\lambda,p^r}(-\xi),s_{\lambda',p^r}(\alpha_n))_{V_{L,\BF}}=\delta_{\lambda,\lambda'}.
\end{equation*}
\end{lemma}

\begin{proof}
By \eqref{eq:exp-E-E+} and \eqref{eq:E+E-E-E+}, we obtain
\begin{align*}
&\?\big(E^{-}(\xi,x_1)\cdots E^{-}(\xi,x_l)\mathbf{1}, E^{-}(-\alpha_n,y_1)\cdots E^{-}(-\alpha_n,y_k)\mathbf{1}\big)_{V_{L,\BF}} \\
&=\big( \mathbf{1}, E^{+}(\xi,x_1^{-1})\cdots E^{+}(\xi,x_l^{-1})E^{-}(-\alpha_n,y_1)\cdots E^{-}(-\alpha_n,y_k)\mathbf{1}\big)_{V_{L,\BF}} \\
&=\prod_{1\le i\le l,1\le j\le k}(1-x_iy_j)^{-\langle\xi,\alpha_n\rangle} \\
&=\prod_{1\le i\le l,1\le j\le k}(1-x_iy_j)^{-p^r} \\
&=\prod_{1\le i\le l,1\le j\le k}\Bigl(\sum_{q\in \BN}x_i^{qp^r}y_j^{qp^r}\Bigr).
\end{align*}
That is,
\begin{align*}
&\?\sum
(s_{-\xi,t_1p^r}\ldots s_{-\xi,t_lp^r},s_{\alpha_n,m_1p^r}\ldots s_{\alpha_n,m_kp^r})_{V_{L,\BF}}x_1^{t_1p^r}\ldots x_l^{t_lp^r}y_1^{m_1p^r}\ldots y_{k}^{m_kp^r}\nonumber\\
&=\prod_{1\le i\le l,1\le j\le k}\Bigl(\sum_{q\in \BN}x_i^{qp^r}y_j^{qp^r}\Bigr).
\end{align*}
Then the desired result follows from the same argument as in 
the proof of \cite[Proposition 3.6]{DG}.
\end{proof}

We finally obtain the following result.
\begin{theorem}\label{th:I-L}
Let $L$ be the root lattice of type $A_n$ with $n+1\equiv 0\pmod{p}$.
Then $\Rad(\cdot,\cdot)_{V_{L,\BF}}=I_{L}$. 
\end{theorem}

\begin{proof}
Lemmas \ref{lem:4.4} and \ref{lem:I-L-ideal} show that $I_{L}$ is an 
ideal contained in $\Rad(\cdot,\cdot)_{V_{L,\BF}}$.
Now we prove that $I_{L}$ in fact equals $\Rad(\cdot,\cdot)_{V_{L,\BF}}$.

By Proposition \ref{prop:s-chur-basis}, we see that
$V_{L,\BF}/I_{L}$ is spanned by the elements in \eqref{eq:V-L-quo-bas-3}.
Assume that
\begin{equation}\label{eq:sum-rad}
	\sum_{\lambda^1,\ldots,\lambda^n\in \mathcal{P},\alpha\in L}c_{\lambda^1,\ldots,\lambda^n}^{\alpha}s_{\lambda^1}(\alpha_1)\cdots s_{\lambda^{n-1}}(\alpha_{n-1})s_{\lambda^n,p^r}(\alpha_n)\iota(e_{\alpha})\in \Rad(\cdot,\cdot)_{V_{L,\BF}},
\end{equation}
where $c_{\lambda^1,\ldots,\lambda^n}^{\alpha}\in \BF$.
For $\lambda=(\lambda_1,\ldots,\lambda_{t})\in \mathcal{P}$, 
we write $|\lambda|=\lambda_1+\cdots+\lambda_t$.
For fixed $\alpha\in L$ and $m\in\BN$, the element
\begin{equation}\label{eq:L-times-Z-ele}
\sum_{\substack{\lambda^1,\ldots,\lambda^n\in \mathcal{P}\\|\lambda^1|+\cdots+|\lambda^{n-1}|+p^r|\lambda^n|=m}}c_{\lambda^1,\ldots,\lambda^n}^{\alpha}s_{\lambda^1}(\alpha_1)\cdots s_{\lambda^{n-1}}(\alpha_{n-1})s_{\lambda^n,p^r}(\alpha_n)\iota(e_{\alpha})
\end{equation}
is homogeneous with respect to the $L\times \BN$-grading.
By Lemma \ref{lem:some-facts-Rad}(ii) and \eqref{eq:sum-rad}, we have
\begin{equation*}
\sum_{\substack{\lambda^1,\ldots,\lambda^n\in \mathcal{P}\\|\lambda^1|+\cdots+|\lambda^{n-1}|+p^r|\lambda^n|=m}}c_{\lambda^1,\ldots,\lambda^n}^{\alpha}s_{\lambda^1}(\alpha_1)\cdots s_{\lambda^{n-1}}(\alpha_{n-1})s_{\lambda^n,p^r}(\alpha_n)\iota(e_{\alpha})\in \Rad(\cdot,\cdot)_{V_{L,\BF}}
\end{equation*}
for $m\in \BN$ and $\alpha\in L$.

For each $\lambda\in \mathcal{P}$, since $s_{\lambda,p^r}(-\xi)$ 
is a well-defined element of $V_{L,\BF}$, 
it follows from Lemma \ref{DG3.3}(i) that $s_{\lambda,p^r}(-\xi)$ 
is a linear combination of elements in $\mathcal{S}$.
Then by Lemma \ref{lem:some-facts-Rad}(i), we see that
$\Rad(\cdot,\cdot)_{V_{L,\BF}}$ is closed under the action of $r_{\lambda,p^r}(-\xi)$.
Furthermore, by Lemma \ref{lem:3.11} and \eqref{eq:r-s-close}, we have
\begin{equation*}
	\delta_{\lambda,\lambda'}=(s_{\lambda,p^r}(-\xi), s_{\lambda',p^r}(\alpha_n))_{V_{L,\BF}}=(\mathbf{1},r_{\lambda,p^r}(-\xi)s_{\lambda',p^r}(\alpha_n))_{V_{L,\BF}}
\end{equation*}
for $\lambda,\lambda'\in \mathcal{P}$.
For $\lambda,\lambda'\in \mathcal{P}$ with $|\lambda|=|\lambda'|$, 
since $\deg r_{\lambda,p^r}(-\xi)=-p^r|\lambda|$ and $\deg s_{\lambda',p^r}(\alpha_n)=p^r|\lambda'|$, 
we have $r_{\lambda,p^r}(-\xi)s_{\lambda',p^r}(\alpha_n)\in \BF\mathbf{1}$, 
and thus
\begin{equation}\label{eq:r-lam-s-lam'=del-lam-lam'}
	r_{\lambda,p^r}(-\xi)s_{\lambda',p^r}(\alpha_n)=\begin{cases}
		\mathbf{1}&\text{if }\lambda=\lambda',\\
		0&\text{if }\lambda\ne\lambda'.
	\end{cases}
\end{equation}
Choose a partition $\lambda^n_{\max}$ of maximal size among the partitions $\lambda^n$
appearing in \eqref{eq:L-times-Z-ele}.
Since $\langle \xi,\alpha_i\rangle=0$ for $i=1,\ldots, n-1$, 
and $r_{-\xi,l}\iota(e_{\alpha})=0$ for $l\in \BZ_+$,
it follows from \eqref{eq:rs-sr} and 
\eqref{eq:r-lam-s-lam'=del-lam-lam'} that
\begin{align*}
	r_{\lambda^n_{\max},p^{r}}(-\xi)\Big(\sum_{\substack{\lambda^1,\ldots,\lambda^n\in \mathcal{P}\\|\lambda^1|+\cdots+|\lambda^{n-1}|+p^r|\lambda^n|=m}}c_{\lambda^1,\ldots,\lambda^n}^{\alpha}s_{\lambda^1}(\alpha_1)\cdots s_{\lambda^{n-1}}(\alpha_{n-1})s_{\lambda^n,p^r}(\alpha_n)\iota(e_{\alpha})\Big)\\
=\sum_{\substack{\lambda^1,\ldots,\lambda^{n-1}\in \mathcal{P},\lambda^n=\lambda^n_{\max}\\|\lambda^1|+\cdots+|\lambda^{n-1}|+p^r|\lambda^n|=m
}}c_{\lambda^1,\ldots,\lambda^n}^{\alpha}s_{\lambda^1}(\alpha_1)\cdots s_{\lambda^{n-1}}(\alpha_{n-1})\iota(e_{\alpha})\in \Rad(\cdot,\cdot)_{V_{L,\BF}}.
\end{align*}
By Lemma \ref{lem:some-facts-Rad}(iii), we have
\begin{equation}\label{eq:sum-in-Rad}
\sum_{\substack{\lambda^1,\ldots,\lambda^{n-1}\in \mathcal{P},\lambda^n=\lambda^n_{\max}\\|\lambda^1|+\cdots+|\lambda^{n-1}|+p^r|\lambda^n|=m
}}c_{\lambda^1,\ldots,\lambda^n}^{\alpha}s_{\lambda^1}(\alpha_1)\cdots s_{\lambda^{n-1}}(\alpha_{n-1})\in \Rad(\cdot,\cdot)_{V_{L,\BF}}.
\end{equation}
Let $L'=\BZ\alpha_1\oplus \cdots \oplus \BZ\alpha_{n-1} (\subset L)$, 
and let $(\cdot,\cdot)_{V_{L',\BF}}$ be the restriction 
of $(\cdot,\cdot)_{V_{L,\BF}}$ to $V_{L',\BF}$.
Note that
\begin{equation*}
	\sum_{\substack{\lambda^1,\ldots,\lambda^{n-1}\in \mathcal{P},\lambda^n=\lambda^n_{\max}\\|\lambda^1|+\cdots+|\lambda^{n-1}|+p^r|\lambda^n|=m
	}}c_{\lambda^1,\ldots,\lambda^n}^{\alpha}s_{\lambda^1}(\alpha_1)\cdots s_{\lambda^{n-1}}(\alpha_{n-1})\in V_{L',\BF}.
\end{equation*}
As $V_{L',\BF}$ is a subalgebra of $V_{L,\BF}$, 
we deduce from \eqref{eq:sum-in-Rad} that
\begin{equation}\label{eq:m-gad-in-M(1)-rad}
\sum_{\substack{\lambda^1,\ldots,\lambda^{n-1}\in \mathcal{P},\lambda^n=\lambda^n_{\max}\\|\lambda^1|+\cdots+|\lambda^{n-1}|+p^r|\lambda^n|=m
}}c_{\lambda^1,\ldots,\lambda^n}^{\alpha}s_{\lambda^1}(\alpha_1)\cdots s_{\lambda^{n-1}}(\alpha_{n-1})\in \Rad(\cdot,\cdot)_{V_{L',\BF}}.
\end{equation}
Since $\det G_{L'}=n\not\equiv 0\pmod{p}$, 
it follows from \cite[Theorem 5]{Mu} that 
\begin{equation}\label{eq:rad=0}
\Rad(\cdot,\cdot)_{V_{L',\BF}}=0.
\end{equation}
Combining \eqref{eq:m-gad-in-M(1)-rad} with \eqref{eq:rad=0}, we have
\begin{equation*}
		\sum_{\substack{\lambda^1,\ldots,\lambda^{n-1}\in \mathcal{P},|\lambda^n|=|\lambda^n_{\max}|\\|\lambda^1|+\cdots+|\lambda^{n-1}|+p^r|\lambda^n|=m
	}}c_{\lambda^1,\ldots,\lambda^n}^{\alpha}s_{\lambda^1}(\alpha_1)\cdots s_{\lambda^{n-1}}(\alpha_{n-1})\iota(e_{\alpha})=0.
\end{equation*}
Since there are only finitely many $\lambda^n\in\mathcal{P}$ with $p^r|\lambda^n|\le m$,
we may repeat this argument finitely many times and finally obtain
\begin{equation*}
	\sum_{\substack{\lambda^1,\ldots,\lambda^n\in \mathcal{P}\\|\lambda^1|+\cdots+|\lambda^{n-1}|+p^r|\lambda^n|=m}
	}c_{\lambda^1,\ldots,\lambda^n}^{\alpha}s_{\lambda^1}(\alpha_1)\cdots s_{\lambda^{n-1}}(\alpha_{n-1})s_{\lambda^n,p^r}(\alpha_n)\iota(e_{\alpha})=0
\end{equation*} 
for $m\in \BN$ and $\alpha\in L$.
Consequently,
\begin{equation*}
	\sum_{\lambda^1,\ldots,\lambda^n\in \mathcal{P},\alpha\in L
	}c_{\lambda^1,\ldots,\lambda^n}^{\alpha}s_{\lambda^1}(\alpha_1)\cdots s_{\lambda^{n-1}}(\alpha_{n-1})s_{\lambda^n,p^r}(\alpha_n)\iota(e_{\alpha})=0.
\end{equation*} 
Therefore, $\Rad(\cdot,\cdot)_{V_{L,\BF}}=I_{L}$.
\end{proof}

We end this subsection by considering the irreducible modules 
of $V_{L,\BF}/\Rad(\cdot,\cdot)_{V_{L,\BF}}$.
For $\beta\in L^*$, we analogously define a symmetric bilinear 
form $(\cdot,\cdot)_{\beta}$ on $V_{\beta+L,\BF}$ by the following conditions:
\begin{enumerate}[(1)]
\item $(\iota(e_{\beta+\alpha}), \iota(e_{\beta+\alpha'}))_{\beta}=\delta_{\alpha+\alpha',0}$ for $\alpha,\alpha'\in L$.
\item $(s_{\alpha,m}w,w')_{\beta}=(w,r_{\alpha,m}w')_{\beta}$ for $w,w'\in V_{\beta+L,\BF}$, $\alpha\in L$, and $m\in \BN$.
\end{enumerate}
Then $\Rad(\cdot,\cdot)_{\beta}$ is defined in the same way.
Under the natural vector-space identification 
\begin{align*}
    T_\beta:~~~~~V_{L,\BF}&\to V_{\beta+L,\BF}\\
    h\iota(e_\alpha)&\mapsto h\iota(e_{\beta+\alpha}),
\end{align*}
where $h\in M(1)_\BF$ and $\alpha\in L$,
the same arguments as for $\Rad(\cdot,\cdot)_{V_{L,\BF}}$ yield the following result.

\begin{proposition}\label{pro:V-beta-Rad}
\begin{enumerate}[(i)]
\item For $\beta\in L^*$, $\Rad(\cdot,\cdot)_{\beta}$ is spanned by elements of the form
\begin{equation*}
    s_{\alpha^1,m_1}\ldots s_{\alpha^k,m_k}s_{\eta,m}\iota(e_{\beta+\alpha}),
\end{equation*}
where $\alpha\in L$, $\alpha^i\in \Pi$, $m_i\in \BZ_+$, $k\in \BN$, and $m\in \BZ_+$ with $p^r\nmid m$.

\item $\Rad(\cdot,\cdot)_{\beta}$ is the $V_{L,\BF}$-submodule of $V_{\beta+L,\BF}$ 
generated by $\{s_{\eta,m}\iota(e_{\beta})\mid m\in \BZ_+\text{ with }p^r\nmid m\}$.
\end{enumerate}
\end{proposition}

Let $L$ be the root lattice of type $A_n$, and let $(b_{ij})$ 
be the adjugate matrix of $G_{L}$. 
For $1\le i\le n$, define
\begin{equation}\label{eq:beta^i}
    \beta^i=\frac{1}{\det G_{L}}\sum_{j=1}^{n}b_{ij}\alpha_{j}.
\end{equation}
By the proof of \cite[Theorem 5]{Mu}, 
the elements $\beta^1, \ldots,\beta^{n}$ form a basis of $L^*$ 
dual to $\alpha_1,\ldots,\alpha_n$.
Hence 
\begin{equation}\label{eq:eta-beta^i}
    \langle\eta, \beta^i\rangle=i\quad\text{ for }1\le i\le n.
\end{equation}
On the other hand, by \eqref{eq:eta-alpha_i},
we have 
\begin{equation}\label{eq:eta-alpha}
    n+1\mid \langle\eta,\alpha\rangle\quad\text{ for every }\alpha\in L.
\end{equation}
Hence, $p^r\mid\langle\eta,\alpha\rangle$ for every $\alpha\in L$.
Consequently, for each coset $\beta+L\in L^*/L$, the condition
\begin{equation*}
p^r\mid\langle\eta,\beta\rangle
\end{equation*}
is independent of the choice of the representative $\beta$.

\begin{lemma}\label{lem:coset-reps-L}
$\{0,\beta^1,\ldots,\beta^n\}$ is a complete set of coset 
representatives for $L^*/L$.
\end{lemma}

\begin{proof}
Note that $\beta^i\notin L$ for all $1\le i\le n$.
If $\beta^i-\beta^j\in L$ for some $1\le i,j\le n$,
then by \eqref{eq:eta-beta^i} and \eqref{eq:eta-alpha}, we have $i=j$.
Therefore, $0,\beta^1, \ldots,\beta^{n}$ represent $n+1$ 
distinct cosets in $L^*/L$.
On the other hand, as $|L^*/L|=\det G_L= n+1$, the assertion holds.
\end{proof}

Denote 
\begin{equation}\label{eq:W-i}
   W_0=V_{L,\BF}/\Rad(\cdot,\cdot)_{V_{L,\BF}} \quad\text{ and }\quad W_i=V_{\beta^i+L,\BF}/\Rad(\cdot,\cdot)_{\beta^i} \text{ for }1\le i\le n,
\end{equation}
where $\beta^i$ are defined in \eqref{eq:beta^i}.
Using the arguments of \cite[Proposition 3.6 and Theorem 3.8]{ZM}, 
together with Lemma \ref{lem:coset-reps-L}, we obtain the following result.

\begin{theorem}\label{th:certain-irre-mod}
The modules $W_i$, where $i=0,1,\ldots,n$, are pairwise inequivalent
irreducible $V_{L,\mathbb F}$-modules.
\end{theorem}

\begin{theorem}\label{th:pr-mid-beta-quot-mod}
The modules $W_{jp^r}$, where $j=0,1,\ldots,a-1$, are pairwise
inequivalent irreducible modules over the quotient vertex algebra
$V_{L,\mathbb F}/\Rad(\cdot,\cdot)_{V_{L,\mathbb F}}$.
\end{theorem}

\begin{proof}
By Theorem \ref{th:certain-irre-mod}, it suffices to show that $W_{jp^r}$ is a 
$V_{L,\mathbb F}/\Rad(\cdot,\cdot)_{V_{L,\mathbb F}}$-module for every $0<j\le a-1$.
Let $m\in \BZ_+$ with $p^r\nmid m$.
By \eqref{eq:eta-beta^i}, we have $\langle \eta,\beta^{jp^r}\rangle=jp^r$.
Using \eqref{eq:Y-s-alpha-n} and the fact that 
$r_{\eta,k}\iota(e_{\beta^{jp^r}})=0$ for $k\in \BZ_+$, we obtain
\begin{align*}
Y(s_{\eta,m},x)\iota(e_{\beta^{jp^r}})&=\sum_{k\in \BN}\binom{jp^r+k}{m}s_{\eta,k}E^{-}(\eta,x)\iota(e_{\beta^{jp^r}})x^{k-m}\nonumber\\
&=\sum_{k,i\in \BN}\binom{jp^r+k}{m}s_{\eta,k}s_{-\eta,i}\iota(e_{\beta^{jp^r}})x^{k-m+i}.
\end{align*}
Thus, for every $l\in \BZ$,
\begin{equation}\label{eq:s-eta-m-l-1}
	(s_{\eta,m})_{-l-1}\iota(e_{\beta^{jp^r}})=\sum_{k=0}^{l+m}\binom{jp^r+k}{m}s_{\eta,k}s_{-\eta,l-k+m}\iota(e_{\beta^{jp^r}}).
\end{equation}
For each term in \eqref{eq:s-eta-m-l-1} with $p^r\nmid k$, 
Lemma \ref{DG3.3}(i) shows that $s_{-\eta,l-k+m}$ is a linear combination of 
elements of $\mathcal{S}$.
Then by Proposition \ref{pro:V-beta-Rad}(ii), 
we have
\begin{equation*}
s_{\eta,k}s_{-\eta,l-k+m}\iota(e_{\beta^{jp^r}})\in \Rad(\cdot,\cdot)_{\beta^{jp^r}}.
\end{equation*}
For every term in \eqref{eq:s-eta-m-l-1} with $p^r\mid k$, we obtain $p^r\mid jp^r+k$.
Since $p^r\nmid m$, Lucas' theorem gives $\binom{jp^r+k}{m}\equiv 0\pmod {p}$.
Thus,
\begin{equation*}
\binom{jp^r+k}{m}s_{\eta,k}s_{-\eta,l-k+m}\iota(e_{\beta^{jp^r}})=0\in \Rad(\cdot,\cdot)_{\beta^{jp^r}}.
\end{equation*}
Then, 
\begin{equation*}
    Y(s_{\eta,m},x)\iota(e_{\beta^{jp^r}})\in \Rad(\cdot,\cdot)_{\beta^{jp^r}}((x)).
\end{equation*}
Since $\{s_{\eta,m}\mid m\in\BZ_+\text{ with }p^r\nmid m\}$ 
generates $\Rad(\cdot,\cdot)_{V_{L,\BF}}$ $(=I_L)$ as an ideal of $V_{L,\BF}$,
it follows that $Y(v,x)\iota(e_{\beta^{jp^r}})\in \Rad(\cdot,\cdot)_{\beta^{jp^r}}((x))$
for all $v\in \Rad(\cdot,\cdot)_{V_{L,\BF}}$.
Since \cite[Theorem 3.1]{ZM} shows that
$V_{\beta^{jp^r}+L,\BF}$ is generated by $\iota(e_{\beta^{jp^r}})$
as a $V_{L,\BF}$-module,
we have $Y(v,x)V_{\beta^{jp^r}+L,\BF} \subset \Rad(\cdot,\cdot)_{\beta^{jp^r}}((x))$
for all $v\in \Rad(\cdot,\cdot)_{V_{L,\BF}}$.
Thus $W_{jp^r}$
is a $V_{L,\BF}/\Rad(\cdot,\cdot)_{V_{L,\BF}}$-module.
\end{proof}

\begin{remark}\label{Rmk:quot-mod}
Now we consider $W_i$ for $p^r\nmid i$. Using \eqref{eq:eta-beta^i},
a direct calculation shows that
\begin{equation*}
(s_{\eta,i})_{i-1}\iota(e_{\beta^i})=\iota(e_{\beta^i})\notin \Rad(\cdot,\cdot)_{\beta^i}.
\end{equation*}
Thus, the $V_{L,\BF}$-module structure of $W_i$
does not factor through $V_{L,\BF}/\Rad(\cdot,\cdot)_{V_{L,\BF}}$.
\end{remark}

\subsection{The irreducible $\BN$-graded modules for $V_{L,\BF}/\Rad(\cdot,\cdot)_{V_{L,\BF}}$}

Recall that $\Pi=\{\alpha_1,\ldots,\alpha_n\}$ is the set of simple 
roots of $\mathfrak{sl}(n+1,\BF)$ with respect to the fixed Cartan subalgebra $\fh$.
Then $\fh$ admits a basis $\{h_{\alpha_i}\mid \alpha_i\in \Pi\}$, 
where $h_{\alpha_i}=E_{i,i}-E_{i+1,i+1}$, and each $E_{i,j}$ is a 
standard matrix unit in $M_{n+1}(\BF)$.
Moreover, the identity matrix decomposes as
\begin{equation}\label{eq:In}
I_{n+1}=h_{\alpha_1}+2h_{\alpha_2}+\cdots+nh_{\alpha_n}
\end{equation}
over $\BF$. 
Thus, by \eqref{eq:psl},
\begin{equation}
\mathfrak{psl}(n+1,\BF)=\Big( \fh\oplus (\mathop{\bigoplus}\limits_{\alpha\in \Delta}\BF E_{\alpha})\Big) \Big/\BF(h_{\alpha_1}+2h_{\alpha_2}+\cdots+nh_{\alpha_n}).
\end{equation}
Since $\langle\alpha,\alpha\rangle=2$ for all $\alpha\in \Pi$, 
we have $h_{\alpha}=\frac{2\alpha}{\langle\alpha,\alpha\rangle}=\alpha$.
It follows from \eqref{eq:2} that
\begin{equation}\label{eq:3.20}
h_{\alpha_i}(x)=\Res_{x_1}\epsilon(\alpha_i,-\alpha_i)^{-1}[E_{\alpha_i}(x),E_{-\alpha_i}(x_1)]\quad\text{for all } \alpha_i\in\Pi.
\end{equation}
We obtain the following isomorphism.

\begin{theorem}\label{th:iso-2}
Let $L$ be the root lattice of type $A_n$ with $n+1\equiv 0\pmod{p}$. 
Then the vertex algebra $V_{L,\BF}/\Rad(\cdot,\cdot)_{V_{L,\BF}}$ 
is isomorphic to $L_{\widehat{\mathfrak{psl}}(n+1,\BF)}(1,0)$.
\end{theorem}

\begin{proof}
By the construction in the proof of Theorem \ref{th:iso-1}, 
$V_{L,\BF}$ is an $\widehat{\mathfrak{sl}}(n+1,\BF)$-module under the actions \eqref{eq:3.14}.
Furthermore, since $\det G_{L}=n+1\equiv 0\pmod{p}$, 
Lemma \ref{lem:Rad} shows that
$\operatorname{Rad}(\cdot,\cdot)_{V_{L,\BF}}$ is a nontrivial 
$V_{L,\BF}$-submodule of $V_{L,\BF}$, and thus a nontrivial 
$\widehat{\mathfrak{sl}}(n+1,\BF)$-submodule of $V_{L,\BF}$.
Consequently, the quotient $V_{L,\BF} / \operatorname{Rad}(\cdot,\cdot)_{V_{L,\BF}}$ 
is also an $\widehat{\mathfrak{sl}}(n+1,\BF)$-module.

We then show that $V_{L,\BF}/\Rad(\cdot,\cdot)_{V_{L,\BF}}$ is a 
$\widehat{\mathfrak{psl}}(n+1,\BF)$-module.
It is enough to show that 
$(h_{\alpha_1}+2h_{\alpha_2}+\cdots+nh_{\alpha_n})(x)$
acts trivially on the quotient $V_{L,\BF}/\Rad(\cdot,\cdot)_{V_{L,\BF}}$. 
By equations \eqref{eq:3.20}, \eqref{eq:3.14}, and \eqref{eq:2'}, we have
\begin{align}\label{eq:3.22}
&\?(h_{\alpha_1}+2h_{\alpha_2}+\cdots+nh_{\alpha_n})(x)\nonumber\\
&=\sum_{j=1}^{n}j\epsilon(\alpha_j,-\alpha_j)^{-1}\Res_{x_1}[E_{\alpha_j}(x),E_{-\alpha_j}(x_1)]\nonumber\\
&=\sum_{j=1}^{n}j\epsilon(\alpha_j,-\alpha_j)^{-1}\Res_{x_1}[Y(\iota(e_{\alpha_{j}}),x),Y(\iota(e_{-\alpha_{j}}),x_1)]\nonumber\\
&=\sum_{j=1}^{n}jY(\alpha_{j}(-1)\mathbf{1},x)=Y\Big(\sum_{j=1}^{n}js_{\alpha_{j},1}\mathbf{1},x\Big).
\end{align}
Note that $X=(1, 2,\ldots,n)^{T}$ is a nonzero solution of the system 
of homogeneous linear equations $G_{L}X=0$ over $\BF$, as $\det G_{L}\equiv 0\pmod{p}$.
Then by the argument in the proof of \cite[Theorem 5]{Mu},
we have
\begin{equation*}
\sum_{j=1}^{n}js_{\alpha_j,1}\mathbf{1}\in \Rad(\cdot,\cdot)_{V_{L,\BF}}.
\end{equation*}
It follows from \eqref{eq:3.22} that  $(h_{\alpha_1}+2h_{\alpha_2}+\cdots+nh_{\alpha_n})(x)=0$ on $V_{L,\BF}/\Rad(\cdot,\cdot)_{V_{L,\BF}}$.

Since $V_{L,\BF}/\Rad(\cdot,\cdot)_{V_{L,\BF}}$ is 
an $\widehat{\mathfrak{sl}}(n+1,\BF)$-module of level $1$, 
it is also a  $\widehat{\mathfrak{psl}}(n+1,\BF)$-module of level $1$.
Clearly, $V_{L,\BF}/\Rad(\cdot,\cdot)_{V_{L,\BF}}$ is generated 
by $\mathbf{1}+\Rad(\cdot,\cdot)_{V_{L,\BF}}$ as a $\widehat{\mathfrak{psl}}(n+1,\BF)$-module, 
and $a(m)\mathbf{1}=0\in \Rad(\cdot,\cdot)_{V_{L,\BF}}$ for $m\in \BN$ 
and $a\in \mathfrak{psl}(n+1,\BF)$ (by the same argument as in Theorem \ref{th:iso-1}).
Thus, the universal property of $V_{\widehat{\mathfrak{psl}}(n+1,\BF)}(1,0)$ shows 
that there is a $\widehat{\mathfrak{psl}}(n+1,\BF)$-module homomorphism 
$\psi: V_{\widehat{\mathfrak{psl}}(n+1,\BF)}(1,0)\to V_{L,\BF}/\Rad(\cdot,\cdot)_{V_{L,\BF}}$ such that 
\begin{equation*}
\psi(\mathbf{1})=\mathbf{1}+\Rad(\cdot,\cdot)_{V_{L,\BF}}.
\end{equation*}
Then 
\begin{equation}\label{eq:4.59}
\psi\big(E_{\gamma_1}(m_1)\cdots E_{\gamma_k}(m_k)\mathbf{1}\big)=\iota(e_{\gamma_1})_{m_1}\cdots\iota(e_{\gamma_k})_{m_k}\mathbf{1}+\Rad(\cdot,\cdot)_{V_{L,\BF}}
\end{equation}
for $\gamma_i\in\pm \Pi$, $m_i\in \BZ$, and $k\in \BN$.
It follows from \cite[Proposition 5.7.9]{LL} 
that $\psi$ is in fact a vertex algebra homomorphism.
By \eqref{eq:4.59}, we see that $\psi$ is surjective, and thus
\begin{equation*}
V_{\widehat{\mathfrak{psl}}(n+1,\BF)}(1,0)/\ker\psi \cong V_{L,\BF}/\Rad(\cdot,\cdot)_{V_{L,\BF}}
\end{equation*}
as vertex algebras.
Moreover, recall that $L_{\widehat{\mathfrak{psl}}(n+1,\BF)}(1,0)=V_{\widehat{\mathfrak{psl}}(n+1,\BF)}(1,0)/J$ 
is simple, where $J$ is the maximal graded ideal of
$V_{\widehat{\mathfrak{psl}}(n+1,\BF)}(1,0)$.
By the same argument as in Theorem \ref{th:iso-1}, 
$\ker\psi$ is a graded ideal of $V_{\widehat{\mathfrak{psl}}(n+1,\BF)}(1,0)$.
Then $\ker \psi\subset J$.
Lemma \ref{lem:simp-quot-latt} shows that 
$V_{L,\BF}/\Rad(\cdot,\cdot)_{V_{L,\BF}}$ is a simple vertex algebra.
Therefore,  $V_{L,\BF}/\Rad(\cdot,\cdot)_{V_{L,\BF}}$ is 
isomorphic to $L_{\widehat{\mathfrak{psl}}(n+1,\BF)}(1,0)$ as vertex algebras.
\end{proof}

We now consider the irreducible $\BN$-graded modules for $V_{L,\BF}/\Rad(\cdot,\cdot)_{V_{L,\BF}}$ when $n+1=ap$ (i.e., the case $r=1$).
Theorem \ref{th:iso-2} yields the following result, 
which was suggested to us by Professor Ching Hung Lam.

\begin{theorem}\label{th:clas-type-A-quo-kp}
Let $L$ be the root lattice of type $A_n$, 
where $n+1=ap$ for some $a\in \BZ_+$ with $\gcd(a,p)=1$.
Then $V_{L,\BF}/\Rad(\cdot,\cdot)_{V_{L,\BF}}$ has 
exactly $a$ pairwise inequivalent irreducible $\BN$-graded modules, namely
\begin{equation}\label{eq:irr-quo-mod}
	W_0,W_{p},W_{2p},\ldots,W_{(a-1)p},
\end{equation}
where $W_i$ are defined as in \eqref{eq:W-i}.
\end{theorem}

\begin{proof}
Since $L_{\widehat{\mathfrak{psl}}(n+1,\BF)}(1,0)$ is a 
quotient of $V_{\widehat{\mathfrak{psl}}(n+1,\BF)}^{0}(1,0)$, 
it follows from \cite[Theorem 3.3]{DR} that
$A(L_{\widehat{\mathfrak{psl}}(n+1,\BF)}(1,0))$ is a quotient 
algebra of $A(V_{\widehat{\mathfrak{psl}}(n+1,\BF)}^{0}(1,0))$. 
By Theorem \ref{th:iso-2} and the same arguments as in the proof of Lemma \ref{lem:E2=0},
we also have $[E_{\alpha}]^2=0$ in 
$A(L_{\widehat{\mathfrak{psl}}(n+1,\BF)}(1,0))$ for all $\alpha\in \Delta$. 
Consequently, Theorem \ref{th:JLM} shows 
that $A(L_{\widehat{\mathfrak{psl}}(n+1,\BF)}(1,0))$ is isomorphic to 
a quotient of $\mathfrak{u}(\mathfrak{psl}(n+1,\BF))/\langle (E_{\alpha})^2\mid \alpha\in \Delta\rangle$.

We then consider the irreducible modules of 
$\mathfrak{u}(\mathfrak{psl}(n+1,\BF))/\langle(E_{\alpha})^2\mid \alpha\in \Delta\rangle$.
By \eqref{eq:psl} and the fact that $I_{n+1}$ is a 
central element of $\mathfrak{sl}(n+1,\BF)$, we see that 
\begin{equation}\label{eq:4.58'}
\mathfrak{u}(\mathfrak{psl}(n+1,\BF))=\mathfrak{u}(\mathfrak{sl}(n+1,\BF))/\mathfrak{u}(\mathfrak{sl}(n+1,\BF))I_{n+1}.
\end{equation}
Let $L(\mu)$ be an irreducible module of $\mathfrak{u}(\mathfrak{sl}(n+1,\BF))$.
Then Theorem \ref{th:the-class-u(g)-mod} shows that $L(\mu)$ is a 
highest weight module with the highest weight 
$\mu=\sum_{i=1}^{n}\mu_i\omega_i$, where $0\le \mu_i\le p-1$ for all $i$.
As in the proof of Theorem \ref{th:clas-type-ADE}, 
such an $L(\mu)$ is an irreducible module for 
$\mathfrak{u}(\mathfrak{sl}(n+1,\BF))/\langle(E_{\alpha})^2\mid \alpha\in \Delta\rangle$ 
if and only if the condition \eqref{eq:hig-wei-res} holds.
By Table \ref{tab:ADE-data},
 the irreducible modules of 
$\mathfrak{u}(\mathfrak{sl}(n+1,\BF))/\langle(E_{\alpha})^2\mid \alpha\in \Delta\rangle$ 
are exactly
\begin{equation}\label{eq:irr-mod-usl}
L(0), \;L(\omega_1),\;\ldots,\; L(\omega_n).
\end{equation}
Since $\mathfrak{u}(\mathfrak{psl}(n+1,\BF))/\langle (E_{\alpha})^2\mid \alpha\in \Delta\rangle$ 
is a quotient of $\mathfrak{u}(\mathfrak{sl}(n+1,\BF))/\langle(E_{\alpha})^2\mid \alpha\in \Delta\rangle$, 
the irreducible modules of 
$\mathfrak{u}(\mathfrak{psl}(n+1,\BF))/\langle (E_{\alpha})^2\mid \alpha\in \Delta\rangle$ 
must be among those listed in \eqref{eq:irr-mod-usl}.
Moreover, by \eqref{eq:4.58'} and \eqref{eq:In}, a module $L(\mu)$ listed 
in \eqref{eq:irr-mod-usl} gives an irreducible module for  
$\mathfrak{u}(\mathfrak{psl}(n+1,\BF))/\langle (E_{\alpha})^2\mid \alpha\in \Delta\rangle$ 
if and only if 
\begin{equation*}
\mu(I_{n+1})=\sum_{i=1}^{n}i\mu(h_{\alpha_i})=\sum_{i=1}^{n}i\mu_i=0.
\end{equation*}
Since $n+1=ap$, the irreducible modules of $\mathfrak{u}(\mathfrak{psl}(n+1,\BF))/\langle(E_{\alpha})^2\mid \alpha\in \Delta\rangle$ are 
\begin{equation*}
L(0),L(\omega_{p}),L(\omega_{2p}),\ldots, L(\omega_{(a-1)p}).
\end{equation*}
Consequently, by \cite[Propositions 2.10, 2.11]{JLM}, 
the vertex algebra $L_{\widehat{\mathfrak{psl}}(n+1,\BF)}(1,0)$ 
has at most $a$ irreducible $\BN$-graded modules up to inequivalent.

On the other hand, recall that $V_{L,\BF}/\Rad(\cdot,\cdot)_{V_{L,\BF}}$ 
is an $\BN$-graded vertex algebra, and for $\beta+L\in L^*/L$, 
$V_{\beta+L,\BF}$ carries an $\BN$-graded $V_{L,\BF}$-module 
structure (see Theorem \ref{th:clas-type-ADE}).
By the same argument as for $\Rad(\cdot,\cdot)_{V_{L,\BF}}$, 
one can show that $\Rad(\cdot,\cdot)_{\beta}$ is also 
an $\BN$-graded $V_{L,\BF}$-module, and hence
$V_{\beta+L,\BF}/\Rad(\cdot,\cdot)_{\beta}$ is 
an $\BN$-graded $V_{L,\BF}$-module.
Then, by Theorem \ref{th:pr-mid-beta-quot-mod}, 
the modules listed in \eqref{eq:irr-quo-mod} are pairwise 
inequivalent irreducible $\BN$-graded $V_{L,\BF}/\Rad(\cdot,\cdot)_{V_{L,\BF}}$-modules.
Since $V_{L,\BF}/\Rad(\cdot,\cdot)_{V_{L,\BF}}$ is 
isomorphic to $L_{\widehat{\mathfrak{psl}}(n+1,\BF)}(1,0)$ by 
Theorem \ref{th:iso-2}, the desired result follows.
\end{proof}

\section*{Acknowledgments.} 
The authors are grateful to Professor Haisheng Li for valuable discussions. 
We also thank Professor Ching Hung Lam for kindly informing us---during the 
revision and proofreading of our manuscript---of his joint work \cite{GL} with Robert L. Griess, Jr., 
which partially overlaps with the present paper.


\begin{thebibliography}{FLM}
\bibitem[B]{B1}
R. Borcherds, 
Vertex algebras, Kac-Moody algebras, and the Monster,
{\em Proc. Natl. Acad. Sci. U.S.A.} {\bf 83} (1986), no. 10, 3068--3071.

\bibitem[C]{C}
C. Curtis,
Representations of Lie algebras of classical type with applications to linear groups,
{\em J. Math. Mech.} {\bf 9} (1960), 307--326.

\bibitem[D]{D}
C. Dong,
Vertex algebras associated with even lattices,
{\em J. Algebra} {\bf 161} (1993), 245--265.


\bibitem[DG]{DG}
C. Dong, R. Griess, 
Integral forms in vertex operator algebras which are invariant under finite groups, 
{\em J. Algebra} {\bf 365} (2012), 184--198. 

\bibitem[DLM]{DLM}
C. Dong, H.-S. Li, G. Mason,
Certain associative algebras similar to $U(sl_2)$ and Zhu's algebra $A(V_{L})$,
{\em J. Algebra} {\bf 196} (1997), 532--551. 


\bibitem[DR]{DR}
C. Dong, L. Ren,
Representations of vertex operator algebras over an arbitrary field,
{\em J. Algebra} {\bf 403} (2014), 497--516.

\bibitem[FHL]{FHL}
I. Frenkel, Y.-Z. Huang, J. Lepowsky,
On axiomatic approaches to vertex operator algebras and modules,
{\em Mem. Amer. Math. Soc.} {\bf 104} (1993), no. 494, viii+64 pp.

\bibitem[FLM]{FLM} 
I. Frenkel, J. Lepowsky, A. Meurman,
{\em Vertex Operator Algebras and the Monster},
Pure Appl. Math., vol. 134, Academic Press, Inc., Boston, MA, 1988, liv+508 pp.

\bibitem[FZ]{FZ}
I. Frenkel, Y. Zhu,
Vertex operator algebras associated to representations of affine and Virasoro algebras,
{\em Duke Math. J.} {\bf 66} (1992), no. 1, 123--168.

\bibitem[G]{Griess}
R. Griess,
{\em An Introduction to Groups and Lattices: Finite Groups and Positive Definite Rational Lattices},
Adv. Lect. Math. (ALM), vol. 15, International Press, Somerville, MA; 
Higher Education Press, Beijing, 2011, iv+251 pp.

\bibitem[GL]{GL}
R. Griess, C. Lam,
Lattice vertex algebras over a field of positive characteristic: degenerate cases,
arXiv preprint, arXiv:2607.28933, 2026.


\bibitem[H1]{H1995}
J. Humphreys,
{\em Conjugacy Classes in Semisimple Algebraic Groups}, 
Math. Surveys Monogr., vol. 43, Amer. Math. Soc., Providence, RI, 1995.


\bibitem[H2]{H}
J. Humphreys,
{\em Introduction to Lie Algebras and Representation Theory}, 
Grad. Texts in Math., vol. 9, Springer-Verlag, New York, 1972.

\bibitem[JLM]{JLM}
X. Jiao, H.-S. Li, Q. Mu,
Modular Virasoro vertex algebras and affine vertex algebras,
{\em J. Algebra} {\bf 519} (2019), 273--311.

\bibitem[K]{Kac2}
V. Kac, 
{\em Vertex Algebras for Beginners},
Univ. Lecture Ser., vol. 10,
American Mathematical Society, Providence, RI, 1997, viii+141 pp.

\bibitem[LL]{LL} 
J. Lepowsky, H.-S. Li, 
{\em Introduction to Vertex Operator Algebras and Their Representations}, 
Progr. Math., vol. 227, 
Birkh\"{a}user Boston, Inc., Boston, MA, 2004, xiii+318 pp.

\bibitem[L]{L}
H.-S. Li,
Symmetric invariant bilinear forms on vertex operator algebras,
{\em J. Pure Appl. Algebra} {\bf 96} (1994), no. 3, 279--297.


\bibitem[LM1]{LM}
H.-S. Li, Q. Mu,
Heisenberg VOAs over fields of prime characteristic and their representations,
{\em Trans. Amer. Math. Soc.} {\bf 370} (2018), no. 2, 1159--1184.

\bibitem[LM2]{LM2}
H.-S. Li, Q. Mu, 
Symmetric invariant bilinear forms on modular vertex algebras,
{\em J. Algebra} {\bf 513} (2018), 435--465.

\bibitem[LW]{LW}
H.-S. Li, Q. Wang,
On vertex algebras and their modules associated with even lattices,
{\em J. Pure Appl. Algebra} {\bf 213} (2009), no. 6, 1097--1111.   

\bibitem[Mc]{McRae}
R. McRae, 
On integral forms for vertex algebras associated with affine Lie algebras and lattices,
{\em J. Pure Appl. Algebra} {\bf 219} (2015), no. 4, 1236--1257.   

\bibitem[M]{Mu}
Q. Mu,
Lattice vertex algebras over fields of prime characteristic,
{\em J. Algebra} {\bf 417} (2014), 39--51.


\bibitem[SF]{SF}
H. Strade, R. Farnsteiner, 
{\em Modular Lie Algebras and Their Representations},
Monogr. Textbooks Pure Appl. Math., vol. 116,
Marcel Dekker, Inc., New York, 1988, x+301 pp.

\bibitem[ZM]{ZM}
H. Zhao, Q. Mu,
Lattice vertex superalgebras and their modules over fields of prime characteristic,
{\em J. London Math. Soc.} {\bf 114} (2026), e70653.

\bibitem[Z]{Zhu1996}
Y. Zhu,
Modular invariance of characters of vertex operator algebras,
{\em J. Amer. Math. Soc.} {\bf 9} (1996), no. 1, 237--302.
\end{thebibliography}
\end{document}